\documentclass[11pt]{article}\UseRawInputEncoding
\usepackage{amssymb,amsfonts,amsmath,latexsym,epsf,tikz,url,blkarray,booktabs, bigstrut}
\usepackage[usenames,dvipsnames]{pstricks}
\usepackage{pstricks-add}
\usepackage{epsfig}
\usepackage{tikz}
\usepackage{comment}
\usepackage{pst-grad} 
\usepackage{pst-plot} 
\usepackage[space]{grffile} 
\usepackage{etoolbox} 
\makeatletter 
\patchcmd\Gread@eps{\@inputcheck#1 }{\@inputcheck"#1"\relax}{}{}
\makeatother

\newtheorem{theorem}{Theorem}[section]
\newtheorem{proposition}[theorem]{Proposition}

\newtheorem{observation}[theorem]{Observation}
\newtheorem{corollary}[theorem]{Corollary}
\newtheorem{lemma}[theorem]{Lemma}
\newtheorem{remark}[theorem]{Remark}
\newtheorem{example}[theorem]{Example}
\newtheorem{definition}[theorem]{Definition}

\newcommand{\qed}{\hfill $\square$\medskip}

\newcommand{\rank}{{\rm rank}}
\newcommand{\modo}{{\rm mod}}
\newcommand{\diam}{{\rm diam}}

\newcommand{\epn}{{\rm epn}}

\newcommand{\sd}{{\rm sp}}
\newcommand{\esd}{{\rm esp}}
\newcommand{\1}{ \vspace{0.1cm} }
\newcommand{\2}{ \vspace{0.2cm} }

\let\oldenumerate\enumerate
\renewcommand{\enumerate}{
  \oldenumerate
  \setlength{\itemsep}{1pt}
  \setlength{\parskip}{0pt}
  \setlength{\parsep}{0pt}
}

\begin{document}

\title{End Super Dominating Sets in Graphs}

\author{
Saieed Akbari$^{1}$
\and
Nima Ghanbari$^{2,}$\footnote{Corresponding author}
\and
Michael A. Henning$^{3}$
}

\date{\today}

\maketitle

\begin{center}
$^{1}$Department of Mathematical Sciences, Sharif University of Technology, Tehran, Iran\\
$^{2}$Department of Informatics, University of Bergen, P.O. Box 7803, 5020 Bergen, Norway\\
$^{3}$ Department of Mathematics and Applied Mathematics, University of Johannesburg, Auckland Park, 2006 South Africa.\\
\medskip
\medskip
{\tt $^1$s\_akbari@sharif.edu ~~  $^2$Nima.Ghanbari@uib.no ~~ $^3$mahenning@uj.ac.za}
\end{center}


\begin{abstract}
Let $G=(V,E)$ be a simple graph. A dominating set of $G$ is a subset $S\subseteq V$ such that every vertex not in $S$ is adjacent to at least one vertex in $S$. The cardinality of a smallest dominating set of $G$, denoted by $\gamma(G)$, is the domination number of $G$. Two vertices are neighbors if they are adjacent. A super dominating set is a dominating set $S$ with the additional property that every vertex in $V \setminus S$ has a neighbor in $S$ that is adjacent to no other vertex in $V \setminus S$. Moreover if every vertex in $V \setminus S$ has degree at least~$2$, then $S$ is an end super dominating set. The end super domination number is the minimum cardinality of an  end super dominating set. We give applications of end super dominating sets as main servers and temporary servers of networks. We determine the exact value of the end super domination number for specific classes of graphs, and we count the number of end super dominating sets in these graphs. Tight upper bounds on the end super domination number are established, where the graph is modified by vertex (edge) removal and contraction.
\end{abstract}


\medskip
\noindent{\bf AMS Subj.\ Class.:} 05C38, 05C69, 05C90

\newpage
\section{Introduction}

Let $G = (V,E)$ be a  graph with the vertex set $V$ and edge set $E$. Throughout this paper, we consider connected graphs without loops and directed edges.
For each vertex $v\in V$, the set $N(v)=\{u\in V \colon uv \in E\}$ refers to the \textit{open neighbourhood} of $v$ and the set $N[v]=N(v)\cup \{v\}$ refers to the \textit{closed neighbourhood} of $v$ in $G$. The \textit{degree} of~$v$, denoted by $\deg(v)$, is the cardinality of $N(v)$. If $V = \{v_1,v_2,\ldots,v_n\}$, the \textit{adjacency matrix} of $G$ is a square $n \times  n$ matrix $A$ such that its element $a_{ij}$ is one when there is an edge from vertex $v_i$ to vertex $v_j$, and zero when there is no edge. The \textit{rank} of a matrix is defined as  the maximum number of linearly independent row (column)  vectors in the matrix. A vertex $v$ is called \textit{universal vertex}, if $N[v]=V$. The maximum and minimum degrees of vertices of $G$ are denoted by $\Delta$ and $\delta$, respectively. For a set $S \subseteq V$, we define $\overline{S}= V \setminus S$. An \emph{isolated vertex} is a vertex of degree~$0$ in $G$, while a \emph{pendant vertex}, also called a \emph{leaf} in the literature, is a vertex of degree~$1$ in $G$. An \emph{isolate}-\emph{free} graph is a graph that contains no isolated vertex. For $k \ge 1$ an integer, we use the standard notation $[k] = \{1,\ldots,k\}$ and $[k]_0 = [k] \cup \{0\} = \{0,1,\ldots,k\}$.

A set $S \subseteq V$ is a  \textit{dominating set} if every vertex in $\overline{S}$ is adjacent to at least one vertex in $S$. The \textit{domination number} $\gamma(G)$ is the minimum cardinality of a dominating set in $G$. A dominating set with cardinality $\gamma(G)$ is called a \emph{$\gamma$-set} of $G$. Domination and its variants are very well studied in the literature. For recent books on domination in graphs, we refer the reader to~\cite{HaHeHe-20,HaHeHe-21,HaHeHe-22,HeYebook}.

A dominating set $S$ is called a \textit{super dominating set}, abbreviated SD-\emph{set}, of $G$ if for every vertex  $u \in \overline{S}$, there exists $v\in S$ such that $N(v)\cap \overline{S}=\{u\}$. The cardinality of a smallest super dominating set of $G$, denoted by $\gamma_{\sd}(G)$, is the \textit{super domination number} of $G$.  The concept of super domination in graphs was introduced and first studied by Lema\'nska, Swaminathan, Venkatakrishnan, and Zuazua~\cite{Lemans}, and studied further, for example, in~\cite{Dett,Nima,Nima1}. Motivated by the definition of a super dominating set and real life applications to networks, we define next an end super dominating set in a graph.

\begin{definition}\label{Def}
{\rm
A super dominating set $S$ is an \emph{end super dominating set}, abbreviated ESD-\emph{set}, of $G$, if the degree of every vertex in $\overline{S}$ is at least~$2$ in $G$, that is, for every vertex  $u\in \overline{S}$, we have $\deg(u) \ge 2$. The cardinality of a smallest end super dominating set of $G$ is the \emph{end super domination number} of $G$	and is denoted by $\gamma_{\esd}(G)$. A $\gamma_{\esd}$-\emph{set} of $G$ is an ESD-set of $G$ of minimum cardinality $\gamma_{\esd}(G)$.
}
\end{definition}

If $S$ is an ESD-set in a graph $G$, then we say that a vertex $u \in \overline{S}$ is \emph{end super dominated} by a vertex $v \in S$ if $N(v) \cap \overline{S}=\{u\}$.

In this paper, we study the end super domination in graphs. In Section~2, we present some applications of end super dominating sets to servers in a network. In Section~3, we determine the end super domination number of specific classes of graphs, and we obtain general bounds on the end super domination number. In Section~\ref{S:trees}, we characterize the infinite family of trees with samllest possible end super domination number. In Section~5, we present bounds on the number of edges of a graph with given end super domination number. In Section~6, we obtain sharp upper and lower bound for end super domination number of a graph by graph modifications, including edge removal, edge contraction, and vertex removal.  In Section~6, we study the relation between the rank of adjacency matrix of a graph and its end super domination  number. Finally, in Section 7 we compute the number of end super dominating sets of minimum cardinality for specific classes of graphs.

\section{Applications and motivation}

In this section, we focus on applications of end super dominating sets. We begin with the following example:

\begin{example}\label{example2}
	{\normalfont
We can consider an ESD-set $E$ as a set of main servers and temporary servers of a network, and $\overline{E}$ as a set of backup servers of this network.
All kinds of servers can be connected or not. A server is a main server if it is connected to one and only one backup server. A server from $E$ is a temporary server if it is connected to more than one backup server. By the definition of ESD-set, each backup server in $\overline{E}$ is connected to at least two main servers or is connected to at least one main server and is connected to at least another backup or temporary server and this temporary server is connected to another back up server and they keep data safe for each other.
}
\end{example}

Suppose that the number of main servers in Example~\ref{example2} is $n$. Then by our definition, we need at most $n$ backup servers and can manage our system more efficiently. For example, suppose that we have a network of the users and servers. First remove all users. Thereafter, determine the end super domination number of this network related to the minimal end super dominating set.  Now, if two backup servers $b_i$ and $b_j$ are connected, then we can replace them with a better backup server $b_k$ with less spending of money by connecting all the previous servers connected to one of the $b_i$ or $b_j$ to $b_k$. Also, if a backup server is too expensive and risky to keep, then we can conversely put at least two backup servers in replacement of the expensive one and connect them to each other and then decide what main server is better to connect to which one of the new backup servers.

Analogous to our motivation in Example~\ref{example2}, it is also possible to discuss applications of end super dominating sets to the location of fire stations of a city and backup forces, hospitals and backup ambulance stations, and many other situations and examples. To illustrate these cases we determine the end super domination number of the graph $G$ shown in Figure~\ref{non-end-super}, and discuss applications of end super dominating sets for this graph.

\begin{figure}
		\begin{center}
			\psscalebox{0.5 0.5}
{
\begin{pspicture}(0,-7.5685577)(13.197116,2.4885578)
\rput[bl](5.84,-1.9314423){$a_1$}
\rput[bl](5.6,-0.55144227){$a_2$}
\rput[bl](5.28,0.62855774){$a_3$}
\rput[bl](4.88,2.1085577){$a_4$}
\rput[bl](7.02,-1.9314423){$b_1$}
\rput[bl](7.96,-2.7114422){$c_1$}
\rput[bl](6.96,-3.4114423){$d_1$}
\rput[bl](5.78,-3.3714423){$e_1$}
\rput[bl](4.86,-2.6914423){$f_1$}
\rput[bl](7.26,-0.61144227){$b_2$}
\rput[bl](10.08,-2.2914422){$c_2$}
\rput[bl](7.22,-4.7914424){$d_2$}
\rput[bl](5.58,-4.8514423){$e_2$}
\rput[bl](2.78,-2.2514422){$f_2$}
\rput[bl](7.68,0.60855776){$b_3$}
\rput[bl](11.72,-2.3314424){$c_3$}
\rput[bl](7.56,-6.011442){$d_3$}
\rput[bl](5.2,-6.051442){$e_3$}
\rput[bl](1.08,-2.2914422){$f_3$}
\rput[bl](8.0,2.1485577){$b_4$}
\rput[bl](12.84,-2.3114424){$c_4$}
\rput[bl](4.86,-7.491442){$e_4$}
\rput[bl](7.92,-7.4514422){$d_4$}
\rput[bl](0.0,-2.3114424){$f_4$}
\psdots[linecolor=black, dotsize=0.4](5.8,-1.3714423)
\psdots[linecolor=black, dotsize=0.4](7.4,-1.3714423)
\psdots[linecolor=black, dotsize=0.4](8.6,-2.5714424)
\psdots[linecolor=black, dotsize=0.4](7.4,-3.7714422)
\psdots[linecolor=black, dotsize=0.4](5.8,-3.7714422)
\psdots[linecolor=black, dotsize=0.4](4.6,-2.5714424)
\psline[linecolor=black, linewidth=0.08](11.8,-2.5714424)(13.0,-2.5714424)(8.6,-2.5714424)(8.6,-2.5714424)
\psline[linecolor=black, linewidth=0.08](7.4,-1.3714423)(8.6,2.2285578)(8.6,2.2285578)
\psline[linecolor=black, linewidth=0.08](5.8,-1.3714423)(4.6,2.2285578)
\psline[linecolor=black, linewidth=0.08](1.4,-2.5714424)(0.2,-2.5714424)(0.2,-2.5714424)
\psline[linecolor=black, linewidth=0.08](5.8,-3.7714422)(4.6,-7.3714423)
\psline[linecolor=black, linewidth=0.08](7.4,-3.7714422)(8.6,-7.3714423)
\psdots[linecolor=black, dotsize=0.4](4.6,2.2285578)
\psdots[linecolor=black, dotsize=0.4](8.6,2.2285578)
\psdots[linecolor=black, dotsize=0.4](13.0,-2.5714424)
\psdots[linecolor=black, dotsize=0.4](8.6,-7.3714423)
\psdots[linecolor=black, dotsize=0.4](4.6,-7.3714423)
\psdots[linecolor=black, dotsize=0.4](0.2,-2.5714424)
\psdots[linecolor=black, fillstyle=solid, dotstyle=o, dotsize=0.4, fillcolor=white](5.4,-0.17144226)
\psdots[linecolor=black, fillstyle=solid, dotstyle=o, dotsize=0.4, fillcolor=white](7.8,-0.17144226)
\psdots[linecolor=black, fillstyle=solid, dotstyle=o, dotsize=0.4, fillcolor=white](5.4,-4.971442)
\psline[linecolor=black, linewidth=0.08](5.0,-6.171442)(8.2,-6.171442)(11.8,-2.5714424)(11.8,-2.5714424)(11.8,-2.5714424)
\psline[linecolor=black, linewidth=0.08](5.0,1.0285578)(8.2,1.0285578)(8.2,1.0285578)
\psline[linecolor=black, linewidth=0.08](4.6,-2.5714424)(1.0,-2.5714424)(1.0,-2.5714424)
\psline[linecolor=black, linewidth=0.08](4.6,-2.5714424)(5.8,-3.7714422)(7.4,-3.7714422)(8.6,-2.5714424)(7.4,-1.3714423)(5.8,-1.3714423)(4.6,-2.5714424)(4.6,-2.5714424)
\psline[linecolor=black, linewidth=0.08](10.2,-2.5714424)(7.8,-4.971442)(7.8,-4.971442)
\psdots[linecolor=black, dotstyle=o, dotsize=0.4, fillcolor=white](10.2,-2.5714424)
\psdots[linecolor=black, dotstyle=o, dotsize=0.4, fillcolor=white](11.8,-2.5714424)
\psdots[linecolor=black, dotstyle=o, dotsize=0.4, fillcolor=white](7.8,-4.971442)
\psdots[linecolor=black, dotstyle=o, dotsize=0.4, fillcolor=white](8.2,-6.171442)
\psdots[linecolor=black, dotstyle=o, dotsize=0.4, fillcolor=white](5.0,-6.171442)
\psdots[linecolor=black, dotstyle=o, dotsize=0.4, fillcolor=white](3.0,-2.5714424)
\psdots[linecolor=black, dotstyle=o, dotsize=0.4, fillcolor=white](1.4,-2.5714424)
\psdots[linecolor=black, dotstyle=o, dotsize=0.4, fillcolor=white](5.0,1.0285578)
\psdots[linecolor=black, dotstyle=o, dotsize=0.4, fillcolor=white](8.2,1.0285578)
\end{pspicture}
}
		\end{center}
		\caption{Graph $G$} \label{non-end-super}
	\end{figure}

\begin{example}\label{example-new}
{\normalfont
Consider  the graph $G$ in Figure \ref{non-end-super}. By the definition of an ESD-set, the six vertices $a_4$, $b_4$, $c_4$, $d_4$, $e_4$,  and $f_4$ belong to every ESD-set in $G$.  One can easily check that the set $S=\{a_1,b_1,c_1,d_1,e_1,f_1,a_4,b_4,c_4,d_4,e_4,f_4\}$ is an ESD-set of $G$ and $\gamma_{\esd}(G)=12$. Now, consider  Example \ref{example2}. If $G$ is our network, then the set of black vertices includes all main servers and the rest are backup servers.
}
\end{example}

Here we consider Figure \ref{non-end-super} and Example \ref{example-new}. Suppose that this is our network and we want to reduce the number of backup servers by the idea that if there is a path between two backup servers without passing a main server, then we use only one such server and connect it to the corresponding main servers and name them as powerful backup servers. Suppose that each main server can support at most 100 users. Then, one of the  modifications to this network is the same as we see in Figure \ref{non-end-super-modificiation}, where PBS is a powerful backup server, MS is a main server and $\lambda$U is $\lambda$ users connecting to a main server ($1\le \lambda\le 100$), respectively.

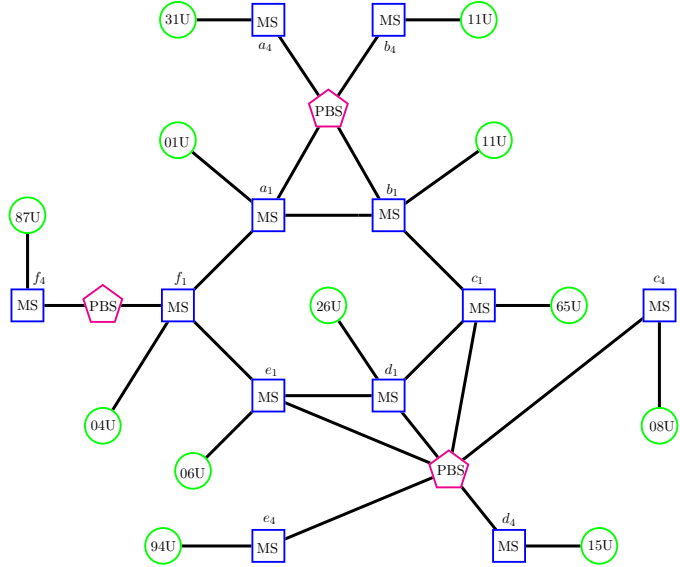
\begin{figure}[!h]
		\begin{center}
			\psscalebox{0.5 0.5}
{
\begin{pspicture}(0,-10.999999)(17.806944,4.006944)
\rput[bl](6.683472,-1.1365284){$a_1$}
\rput[bl](6.643472,2.6634717){$a_4$}
\rput[bl](10.043472,-1.1965283){$b_1$}
\rput[bl](12.3034725,-3.5365283){$c_1$}
\rput[bl](10.003472,-5.976528){$d_1$}
\rput[bl](6.8234725,-5.976528){$e_1$}
\rput[bl](4.4034724,-3.5565283){$f_1$}
\rput[bl](9.983472,2.6234717){$b_4$}
\rput[bl](17.143473,-3.5365283){$c_4$}
\rput[bl](6.763472,-9.916529){$e_4$}
\rput[bl](13.123472,-9.956529){$d_4$}
\rput[bl](0.6434723,-3.5365283){$f_4$}
\psline[linecolor=black, linewidth=0.08](6.9034724,-1.6965283)(9.3034725,-1.6965283)(9.3034725,-1.6965283)
\psline[linecolor=black, linewidth=0.08](9.3034725,-1.6965283)(10.103473,-1.6965283)(10.103473,-1.6965283)
\psline[linecolor=black, linewidth=0.08](6.9034724,-1.6965283)(4.9034724,-3.6965284)(4.5034723,-4.0965285)(6.9034724,-6.496528)(10.103473,-6.496528)(12.503472,-4.0965285)(10.103473,-1.6965283)(10.103473,-1.6965283)
\psline[linecolor=black, linewidth=0.08](6.9034724,-1.6965283)(8.503472,1.1034716)(10.103473,-1.6965283)(10.103473,-1.6965283)
\psline[linecolor=black, linewidth=0.08](6.9034724,-6.496528)(11.703472,-8.496529)(10.103473,-6.496528)(10.103473,-6.496528)
\psline[linecolor=black, linewidth=0.08](12.503472,-4.0965285)(11.703472,-8.496529)(11.703472,-8.496529)
\psline[linecolor=black, linewidth=0.08](4.5034723,-4.0965285)(2.5034723,-4.0965285)(2.5034723,-4.0965285)
\psline[linecolor=black, linewidth=0.08](8.503472,1.1034716)(6.9034724,3.5034716)(6.9034724,3.5034716)
\psline[linecolor=black, linewidth=0.08](8.503472,1.1034716)(10.103473,3.5034716)(10.103473,3.5034716)
\psline[linecolor=black, linewidth=0.08](17.303473,-4.0965285)(11.703472,-8.496529)(13.3034725,-10.496529)(13.3034725,-10.496529)
\psline[linecolor=black, linewidth=0.08](11.703472,-8.496529)(6.9034724,-10.496529)(6.9034724,-10.496529)
\psline[linecolor=black, linewidth=0.08](2.5034723,-4.0965285)(0.50347227,-4.0965285)(0.50347227,-4.0965285)
\psline[linecolor=black, linewidth=0.08](10.103473,3.5034716)(12.503472,3.5034716)(12.503472,3.5034716)
\psline[linecolor=black, linewidth=0.08](6.9034724,3.5034716)(4.5034723,3.5034716)(4.5034723,3.5034716)
\psline[linecolor=black, linewidth=0.08](0.50347227,-4.0965285)(0.50347227,-1.6965283)(0.50347227,-1.6965283)
\psline[linecolor=black, linewidth=0.08](17.303473,-4.0965285)(17.303473,-7.2965283)(17.303473,-7.2965283)
\psline[linecolor=black, linewidth=0.08](13.3034725,-10.496529)(15.703472,-10.496529)(15.703472,-10.496529)
\psline[linecolor=black, linewidth=0.08](6.9034724,-10.496529)(4.103472,-10.496529)(4.103472,-10.496529)
\psline[linecolor=black, linewidth=0.08](4.5034723,-4.0965285)(2.5034723,-7.2965283)(2.5034723,-7.2965283)
\psline[linecolor=black, linewidth=0.08](6.9034724,-1.6965283)(4.5034723,0.30347168)(4.5034723,0.30347168)
\psline[linecolor=black, linewidth=0.08](10.103473,-1.6965283)(12.903472,0.30347168)(12.903472,0.30347168)
\psline[linecolor=black, linewidth=0.08](12.503472,-4.0965285)(14.903472,-4.0965285)(14.903472,-4.0965285)
\psline[linecolor=black, linewidth=0.08](10.103473,-6.496528)(8.503472,-4.0965285)(8.503472,-4.0965285)
\psline[linecolor=black, linewidth=0.08](6.9034724,-6.496528)(4.9034724,-8.496529)(4.9034724,-8.496529)
\psdots[linecolor=green, dotstyle=o, dotsize=1.0, fillcolor=white](8.503472,-4.0965285)
\psdots[linecolor=green, dotstyle=o, dotsize=1.0, fillcolor=white](12.903472,0.30347168)
\psdots[linecolor=green, dotstyle=o, dotsize=1.0, fillcolor=white](12.503472,3.5034716)
\psdots[linecolor=green, dotstyle=o, dotsize=1.0, fillcolor=white](4.5034723,3.5034716)
\psdots[linecolor=green, dotstyle=o, dotsize=1.0, fillcolor=white](4.5034723,0.30347168)
\psdots[linecolor=green, dotstyle=o, dotsize=1.0, fillcolor=white](0.50347227,-1.6965283)
\psdots[linecolor=green, dotstyle=o, dotsize=1.0, fillcolor=white](2.5034723,-7.2965283)
\psdots[linecolor=green, dotstyle=o, dotsize=1.0, fillcolor=white](4.9034724,-8.496529)
\psdots[linecolor=green, dotstyle=o, dotsize=1.0, fillcolor=white](4.103472,-10.496529)
\psdots[linecolor=green, dotstyle=o, dotsize=1.0, fillcolor=white](15.703472,-10.496529)
\psdots[linecolor=green, dotstyle=o, dotsize=1.0, fillcolor=white](17.303473,-7.2965283)
\psdots[linecolor=green, dotstyle=o, dotsize=1.0, fillcolor=white](14.903472,-4.0965285)
\psdots[linecolor=magenta, dotstyle=pentagon, dotsize=1.0, fillcolor=white](8.503472,1.1034716)
\psdots[linecolor=magenta, dotstyle=pentagon, dotsize=1.0, fillcolor=white](11.703472,-8.496529)
\psdots[linecolor=magenta, dotstyle=pentagon, dotsize=1.0, fillcolor=white](2.5034723,-4.0965285)
\psdots[linecolor=blue, dotstyle=square, dotsize=1.0, fillcolor=white](6.9034724,3.5034716)
\psdots[linecolor=blue, dotstyle=square, dotsize=1.0, fillcolor=white](10.103473,3.5034716)
\psdots[linecolor=blue, dotstyle=square, dotsize=1.0, fillcolor=white](10.103473,-1.6965283)
\psdots[linecolor=blue, dotstyle=square, dotsize=1.0, fillcolor=white](6.9034724,-1.6965283)
\psdots[linecolor=blue, dotstyle=square, dotsize=1.0, fillcolor=white](4.5034723,-4.0965285)
\psdots[linecolor=blue, dotstyle=square, dotsize=1.0, fillcolor=white](6.9034724,-6.496528)
\psdots[linecolor=blue, dotstyle=square, dotsize=1.0, fillcolor=white](10.103473,-6.496528)
\psdots[linecolor=blue, dotstyle=square, dotsize=1.0, fillcolor=white](12.503472,-4.0965285)
\psdots[linecolor=blue, dotstyle=square, dotsize=1.0, fillcolor=white](6.9034724,-10.496529)
\psdots[linecolor=blue, dotstyle=square, dotsize=1.0, fillcolor=white](17.303473,-4.0965285)
\psdots[linecolor=blue, dotstyle=square, dotsize=1.0, fillcolor=white](0.50347227,-4.0965285)
\psdots[linecolor=blue, dotstyle=square, dotsize=1.0, fillcolor=white](13.3034725,-10.496529)
\rput[bl](11.383472,-8.6165285){PBS}
\rput[bl](2.1434722,-4.2565284){PBS}
\rput[bl](8.143473,0.94347167){PBS}
\rput[bl](6.603472,3.3034716){MS}
\rput[bl](9.863472,3.3634717){MS}
\rput[bl](6.603472,-1.8765283){MS}
\rput[bl](9.8434725,-1.7965283){MS}
\rput[bl](12.243472,-4.2965283){MS}
\rput[bl](9.8034725,-6.6765285){MS}
\rput[bl](6.623472,-6.6565285){MS}
\rput[bl](4.223472,-4.2765284){MS}
\rput[bl](6.603472,-10.676528){MS}
\rput[bl](13.003472,-10.636528){MS}
\rput[bl](17.023472,-4.2565284){MS}
\rput[bl](0.2234723,-4.2165284){MS}
\rput[bl](8.183473,-4.2565284){26U}
\rput[bl](12.203472,3.4034717){11U}
\rput[bl](12.583472,0.14347167){11U}
\rput[bl](14.563473,-4.2365284){65U}
\rput[bl](4.143472,0.12347168){01U}
\rput[bl](4.163472,3.4034717){31U}
\rput[bl](0.18347229,-1.8765283){87U}
\rput[bl](2.1834724,-7.416528){04U}
\rput[bl](4.5634723,-8.696528){06U}
\rput[bl](3.7834723,-10.636528){94U}
\rput[bl](17.043472,-7.436528){08U}
\rput[bl](15.403472,-10.6165285){15U}
\end{pspicture}
}
		\end{center}
		\caption{A modification of graph $G$ in Figure \ref{non-end-super} to a network} \label{non-end-super-modificiation}
	\end{figure}

\section{Special classes of graphs and general bounds}

In this section, we present some results on ESD-sets of graphs. A SD-set is not necessarily an ESD-set, since there might be a vertex with degree~$1$ in the complement of our SD-set. As Definition~\ref{Def}, if $G$ is a graph and $S$ is a SD-set of $G$, then $S$ is an ESD-set of $G$, if for all $u\in \overline{S}$, $\deg(u)\ge 2$. Also, we have:

\begin{proposition}\label{prop-2}
If $G$ is a graph, then $\gamma(G)\le \gamma_{\sd}(G)\le \gamma_{\esd}(G)$.
\end{proposition}

Before we continue, we state a known result on the super domination number of a graph.

\begin{theorem}{\rm (\cite{Lemans})}
\label{thm-1}
If $G$ is connected graph of order $n \ge 2$, then the following properties hold. \\ [-24pt]
\begin{enumerate}
\item[{\rm (a)}] $1 \le \gamma(G) \le \frac{n}{2} \le \gamma_{\sd}(G) \le n-1$. \1
\item[{\rm (b)}] $\gamma(G) + \gamma_{\sd}(G) \le n$.
\end{enumerate}
\end{theorem}

As a consequence of Definition~\ref{Def}, Proposition~\ref{prop-2} and Theorem~\ref{thm-1} we have the following remark.

\begin{remark}
\label{thm-mix}
{\normalfont
If $G$ is graph of order $n$, then the following properties hold. \\ [-22pt]
\begin{enumerate}
\item[{\rm (a)}] All vertices of degree at most~$1$ belong to every ESD-set in $G$. \1
\item[{\rm (b)}] If $\delta(G) \ge 2$, then $\gamma_{\esd}(G)=\gamma_{\sd}(G)$. \1
\item[{\rm (c)}] If every component of $G$ has order at least~$3$, then
\[
1 \le \gamma(G) \le \frac{n}{2} \le \gamma_{\sd}(G) \le \gamma_{\esd}(G) \le n-1.
\]
\end{enumerate}
}
\end{remark}




We next determine the end super domination number for certain graph classes. Lema\'nska et al.~\cite{Lemans} determined the super domination number of cycles, complete graphs, and complete bipartite graphs of order at least~$3$. Since these graphs have minimum degree at least~$2$, by Remark~\ref{thm-mix}(b), their end super domination number is equal to their super domination number. This yields the following result. Recall that $K_n$ denotes the complete graph of order~$n$, $K_{n,m}$ denotes the complete bipartite graph with partite sets of cardinalities~$n$ and~$m$, and $C_n$ denotes a cycle of order~$n$.

\begin{proposition}{\rm (\cite{Lemans})}
\label{thm-special}
The following hold. \\ [-22pt]
\begin{enumerate}
\item[{\rm (a)}] For $n \ge 3$, $\gamma_{\esd}(K_n)=n-1$. \1
\item[{\rm (b)}] For $\min\{n,m\}\ge 2$, $\gamma_{\esd}(K_{n,m})=n+m-2$. \1
\item[{\rm (c)}] For $n \ge 3$,
\begin{displaymath}
\gamma_{\esd}(C_n)= \left\{ \begin{array}{ll}
\lceil\frac{n}{2}\rceil & \textrm{if $n \equiv 0, 3 ~ (mod ~ 4)$, }\\
\\
\lceil\frac{n+1}{2}\rceil & \textrm{otherwise.}
\end{array} \right.
\end{displaymath}		
\end{enumerate}
\end{proposition}

We note that if $G \in \{K_1,K_2\}$ has order~$n$, then $\gamma_{\esd}(G)=n$. If $G$ is a star graph $K_{1,n-1}$ of order~$n \ge 3$, then the set of $n-1$ pendant vertices (of degree~$1$) form the unique ESD-set of $G$ and $\gamma_{\esd}(G)=n-1$. We next determine the end super domination number of a path $P_n$ of order~$n$.

\begin{theorem}
\label{thm-path-esp}
For $n \in \mathbb{N} \setminus \{1\}$ and $k \in \mathbb{N} \cup \{0\}$, the following hold.
\[
 	\gamma_{\esd}(P_n)=\left\{
  	\begin{array}{ll}
  	{\displaystyle
  		2k}&
  		\quad\mbox{if $n=4k$, } \2 \\
  		{\displaystyle
  			2k+1}&
  			\quad\mbox{if $n=4k+1$,} \2 \\
  		{\displaystyle
			2k+2}&
  			\quad\mbox{if $n=4k+2$ or $n=4k+3$.}
  				  \end{array}
  					\right.	
\] 	
\end{theorem}
\begin{proof}
We proceed by induction on the order~$n \ge 2$ of a path $P_n$. It is a simple exercise to check that $\gamma_{\esd}(P_2)=2$, $\gamma_{\esd}(P_3)=2$, $\gamma_{\esd}(P_4)=2$, $\gamma_{\esd}(P_5)=3$, $\gamma_{\esd}(P_6)=4$ and $\gamma_{\esd}(P_7)=4$. Hence the assertion holds for $\gamma_{\esd}(P_n)$, where $2 \le n \le 7$. This establishes the base cases. Let $ n\ge 8$ and let $G$ be the path $v_1v_2 \ldots v_n$. We consider four cases.

\medskip
\emph{Case~1: $n=4k$, for some $k\in \mathbb{N} \setminus \{1\}$.} In this case, we let
\[
S = \bigcup_{i=0}^{k-1} \{v_{4i+1},v_{4i+4}\}.
\]
The set $S$ is a dominating set of $G$. Moreover since the ends $v_1$ and $v_{4k}$ of the path belong to $S$, every vertex in $\overline{S}$ has degree~$2$ in $G$. If $v \in \overline{S}$, then either $v = v_{4i+2}$ or $v = v_{4i+3}$ for some $i \in [k-1]_0$. If $v = v_{4i+2}$, then $N(v_{4i+1}) \cap \overline{S} = \{v_{4i+2}\}$, and so the vertex $v \in \overline{S}$ is end super dominated by the vertex $v_{4i+1} \in S$. If $v = v_{4i+3}$, then $N(v_{4i+4}) \cap \overline{S} = \{v_{4i+3}\}$, and so the vertex $v \in \overline{S}$ is end super dominated by the vertex $v_{4i+4} \in S$. Thus, $S$ is an ESD-set of $G$, and so $\gamma_{\esd}(G) \le |S| = 2k$. By Remark~\ref{thm-mix}, $\gamma_{\esd}(G) \ge \lceil \frac{n}{2} \rceil = 2k$. Consequently, $\gamma_{\esd}(G) = 2k$.

\medskip
\emph{Case~2: $n=4k+1$, for some $k\in \mathbb{N} \setminus \{1\}$.}  In this case, we let
\[
S = \left( \bigcup_{i=0}^{k-1} \{v_{4i+1},v_{4i+4}\} \right) \cup \{v_{4k+1}\}.
\]
Analogous arguments as in Case~1 show that $S$ is an ESD-set of $G$, and so $\gamma_{\esd}(G) \le |S| = 2k+1$. By Remark~\ref{thm-mix}, $\gamma_{\esd}(G) \ge \lceil \frac{n}{2} \rceil = 2k+1$. Consequently, $\gamma_{\esd}(G) = 2k+1$.

\medskip
\emph{Case~3: $n=4k+2$, for some $k\in \mathbb{N} \setminus \{1\}$.}  Let $S$ be a $\gamma_{\esd}$-set of $G$, and so $S$ is an ESD-set of $G$ and $|S| = \gamma_{\esd}(G)$. We show that $|S| \ge k+2$. By Remark \ref{thm-mix}, we note that $\{v_1,v_{4k+2}\} \subseteq S$. If $v_2 \in S$, then the set $S' = S \setminus \{v_1\}$ is an ESD-set of the path $G-v_1 = P_{4k+1}$, and so by induction we have $2k + 1 = \gamma_{\esd}(G-v_1) \le |S'| = |S| - 1$, and so $|S| \ge 2k+2$, as desired. Analogously, if $v_{4k+1} \in S$, then $|S| \ge 2k+2$. Hence we may assume that $\{v_2,v_{4k+1}\} \subseteq \overline{S}$.

Suppose that $v_3 \notin S$. This implies that $v_4 \in S$ and the vertex $v_3$ is end super dominated by the vertex $v_4$. This in turn implies that $v_5 \in S$. In this case, the set $S' = S \setminus \{v_1,v_4\}$ is an ESD-set of the path $G' = G\setminus\{v_1,v_2,v_3,v_4\} = P_{4(k-1)+2}$, and so by induction we have $2(k-1) + 2 = \gamma_{\esd}(G') \le |S'| = |S| - 2$, and so $|S| \ge 2k+2$, as desired. Analogously, if $v_{4k} \notin S$, then $|S| \ge 2k+2$. Hence, we may assume that $\{v_3,v_{4k}\} \subset S$. We now let $G' = G - \{v_1,v_2,v_{4k+1},v_{4k+2}\}$ and we let $S' = S \setminus \{v_1,v_{4k+2}\}$. Since $S$ is an ESD-set of $G$, the set $S'$ is an ESD-set of $G' = P_{4(k-1)+2}$. Thus, by induction, we have $2(k-1) + 2 = \gamma_{\esd}(G') \le |S'| = |S| - 2$, and so $|S| \ge 2k+2$, as desired. Thus in all cases, $|S| \ge 2k+2$, implying that $\gamma_{\esd}(G) = |S| \ge 2k+2$. However, the set
\[
S^* = \left( \bigcup_{i=0}^{k-1} \{v_{4i+1},v_{4i+4}\} \right) \cup \{v_{4k+1},v_{4k+2}\}
\]
is an ESD-set of $G$ using analogous arguments as in Case~1, and so $\gamma_{\esd}(G) \le |S^*| \ge 2k+2$. Consequently, $\gamma_{\esd}(G) = 2k+2$.

\medskip
\emph{Case~4: $n=4k+3$, for some $k\in \mathbb{N} \setminus \{1\}$.} In this case, we let
\[
S = \left( \bigcup_{i=0}^{k-1} \{v_{4i+1},v_{4i+4}\} \right) \cup \{v_{4k+1},v_{4k+3}\}.
\]
Analogous arguments as in Case~1 show that $S$ is an ESD-set of $G$, and so $\gamma_{\esd}(G) \le |S| = 2k+2$. By Remark~\ref{thm-mix}, $\gamma_{\esd}(G) \ge \lceil \frac{n}{2} \rceil = 2k+2$. Consequently, $\gamma_{\esd}(G) = 2k+2$.~\qed
\end{proof}

Recall that a \textit{universal vertex} in a graph $G$ is a vertex adjacent to every other vertex in the graph. A graph is \textit{ $F$-free} if it does not contain $F$ as an induced subgraph. Moreover, a graph is \textit{$(F_1,F_2)$-free} if it is $F_1$-free and $F_2$-free. The following lemma is almost certainly known, but we were unable to find a reference. For completeness, we therefore give the simple proof of this result.

\begin{lemma}
\label{lem-upper-all-universal}
If $G$ is a $(P_4,C_4)$-free connected graph of order $n \ge 3$, then $G$ has a universal vertex.
\end{lemma}
\begin{proof}
We proceed by induction on $n$ to prove that if $G$ is a $(P_4,C_4)$-free connected graph of order~$n \ge 3$, then $G$ has a universal vertex. If $n=3$, then $G$ has a universal vertex, as desired. This establishes the base case. Let $n \ge 4$ and let $G$ be a $(P_4,C_4)$-free graph of order~$n$. Let $T$ be a spanning tree of $G$, and let $v$ be a leaf in $T$. Let $G' = G-v$, and so $G'$ is a $(P_4,C_4)$-free connected graph of order~$n' = n-1$. Applying the inductive hypothesis to $G'$, the graph $G'$ contains a universal vertex $u$. If $uv \in E(G)$, then $u$ is a universal vertex in $G$, as desired. Hence, we may assume that $uv \notin E(G)$. Let $w$ be the (unique) neighbor of $v$ in the tree $T$. Since $u$ is a universal vertex in $G'$, the vertex $w$ is a common neighbor of $u$ and $v$ in $G$. We show that $w$ is a universal vertex in $G$. Suppose, to the contrary, that there is a vertex $x$ not adjacent to $w$ in $G$. Necessarily, $x \notin \{u,v,w\}$. Let $R = \{u,v,w,x\}$. If $xv \notin E(G)$, then $G[R] = P_4$, while if $xv \in E(G)$, then $G[R] = C_4$. In both cases, we contradict the supposition that $G$ is $(P_4,C_4)$-free. Hence, $w$ is a universal vertex in $G$, as desired.~\qed
\end{proof}

By Remark \ref{thm-mix}, if $G$ is a connected graph of order~$n \ge 3$, then $\gamma_{\esd}(G) \le n-1$. We show next that if equality holds, then the graph $G$ necessarily contains a universal vertex.

\begin{theorem}
\label{thm-upper-all-universal}
If $G$ is connected graph of order $n \ge 3$ such that $\gamma_{\esd}(G)=n-1$, then $G$ has a universal vertex.
\end{theorem}
\begin{proof}
Let $G$ be a connected graph of order $n \ge 3$ satisfying $\gamma_{\esd}(G)=n-1$. Suppose, to the contrary, that $G$ does not have a universal vertex. By Lemma~\ref{lem-upper-all-universal}, the graph $G$ contains $P_4$ or $C_4$ as an induced subgraph. Suppose that $F = P_4$ is an induced subgraph of $G$, where $F$ is the path $v_1v_2v_3v_4$. In this case, the set $V(G) \setminus \{v_2,v_3\}$ is an ESD-set of $G$, where the vertices $v_2$ and $v_3$ are end super dominated by the vertices $v_1$ and $v_4$, respectively. Hence, $\gamma_{\esd}(G) \le |V(G) \setminus \{v_2,v_3\}| = n - 2$, a contradiction. Therefore, $F = C_4$ is an induced subgraph of $G$, where $F$ is the cycle $v_1v_2v_3v_4v_1$. As before, the set $V(G) \setminus \{v_2,v_3\}$ is an ESD-set of $G$, and so $\gamma_{\esd}(G) \le n - 2$, a contradiction.~\qed
\end{proof}

For a set $S \subseteq V$ and a vertex $v \in S$ in a graph $G$, the \emph{$S$-external private neighborhood} of $v$, denoted $\epn(v,S)$, is the set of vertices in $\overline{S} = V \setminus S$ that are adjacent to~$v$ but to no other vertex in $S$, that is, if $w \in \epn(v,S)$, then $w \in V \setminus S$ and $N(w) \cap S = \{v\}$. A vertex in the set $\epn(v,S)$ is an \emph{$S$-external private neighbor} of $v$. In 1979 Bollob\'{a}s and Cockayne~\cite{BoCo-79} established the following property of minimum dominating sets in graphs.

\begin{lemma}{\rm (\cite{BoCo-79})}
\label{lem-BC}
Every isolate-free graph $G$ contains a $\gamma$-set $S$ such that $\epn(v,S) \ne \emptyset$ for every vertex $v \in S$.
\end{lemma}

Let $G$ be a graph of order~$n$ with $\delta(G) \ge 2$. If $S$ is a $\gamma$-set of $G$ satisfying the statement of Lemma~\ref{lem-BC}, then the set $\overline{S}$ is an ESD-set of $G$, implying that $\gamma_{\esd}(G) \le |\overline{S}| = n - |S| = n - \gamma(G)$. This yields the following consequence of Theorem~\ref{lem-BC}.

\begin{corollary}
\label{cor-BC}
If $G$ is a graph of order $n$ with $\delta(G) \ge 2$, then $\gamma_{\esd}(G) \le n - \gamma(G)$.
\end{corollary}

We remark that strict inequality in Corollary~\ref{cor-BC} is possible. For example, for $k \ge 2$ if $G$ is obtained from $k$ vertex disjoint copies of $K_3$ by selecting one vertex from each triangle and identifying these $k$ selected vertices into one new vertex, then the resulting graph $G$ has order~$n = 2k+1$. Moreover, $\gamma(G) = 1$ and $\gamma_{\sd}(G) = k+1$, and so $\gamma_{\sd}(G) = k+1 < 2k = n - \gamma(G)$.

If $S$ is a $\gamma$-set of a graph $G$ of order~$n$ with maximum degree~$\Delta$, then every vertex of $S$ dominates itself and at most~$\Delta$ other vertices. Therefore at most $(\Delta + 1)|S|$ distinct vertices are dominated by the dominating set $S$, implying that $n \le (\Delta + 1)|S| = (\Delta + 1)\gamma(G)$. As a consequence of this elementary lower bound on the domination number, first observed by Walikar et al.~\cite{WaAcSa-79}, together with Corollary~\ref{cor-BC}, we have the following upper bound on the end super domination number of a graph in terms of its order and maximum degree.

\begin{corollary}
\label{trivial-bound}
If $G$ is a graph of order~$n$ with maximum degree~$\Delta$, then
\[
\gamma_{\esd}(G) \le \left( \frac{\Delta}{\Delta+1} \right) n.
\]
\end{corollary}

We remark that the upper bound in Corollary~\ref{trivial-bound} is tight. The simplest example of a graph achieving equality in the bound is when $G = K_{\Delta + 1}$ where $\Delta \ge 2$ is an arbitrary integer. In this case, $G$ has order~$n = \Delta + 1$, maximum degree~$\Delta$, and $\gamma_{\esd}(G) = \Delta$, and so $\gamma_{\esd}(G) = \left( \frac{\Delta}{\Delta+1} \right) n$. As a special case of Corollary~\ref{trivial-bound}, for $k \ge 2$ if $G$ is a $k$-regular graph, then $\gamma_{\esd}(G) \le (\frac{k}{k+1})n$.

In 1956 Nordhaus and Gaddum~\cite{Nord} established sharp bounds on the sum and product of the chromatic numbers of a graph and its complement. In 1972 Jaeger and Payan~\cite{JaPa72} presented the first Nordhaus-Gaddum bounds involving the domination numbers of a graph and its complement. As an immediate consequence of Remark~\ref{thm-mix}, we have the following Nordhaus-Gaddum bounds involving end super domination. The upper bounds are achievable, for example, when $G = K_n$.

\begin{theorem}
\label{Nor-Gad}
If $G$ is a graph of order $n \ge 3$, then the following hold. \\ [-22pt]
\begin{enumerate}
\item[{\rm (a)}] $n \le \gamma_{\esd}(G) + \gamma_{\esd}(\overline{G}) \le 2n-1$. \1
\item[{\rm (b)}] $\frac{n^2}{4}\le  \gamma_{\esd}(G) \cdot \gamma_{\esd}(\overline{G}) \le n(n-1)$.
\end{enumerate}
\end{theorem}

\section{Trees with small end super domination number}
\label{S:trees}

By Remark~\ref{thm-mix}, if $T$ is a tree of order~$n \ge 3$, then $\gamma_{\esd}(T) \ge \frac{1}{2}n$. In this section, we characterize the trees achieving equality in this lower bound. For this purpose, we define a \emph{$2$-coloring} of a tree $T$ as a coloring of the vertices of $T$, one color to each vertex, using the colors amber and blue. We denote by $A$ and $B$ the resulting sets of vertices in $T$ colored amber and blue, respectively.

\begin{definition}{\rm (The family ${\cal T}$)}
\label{fam-cT}
{\rm
Let ${\cal T}$ of the family of $2$-colored trees that can be obtained from a sequence $T_1, \ldots, T_k$, $k \ge 1$, of trees where $T_1$ (called the \emph{$2$-colored base tree}) is the path $P_4$ where the two leaves are colored blue and the two vertices of degree~$2$ are colored amber, and if $k \ge 2$, then the $2$-colored tree $T_{i+1}$ can be obtained from the $2$-colored tree $T_i$ by applying Operation ${\cal O}_1$ or Operation ${\cal O}_2$ defined below, for all $i \in [k-1]$.}

\begin{itemize}
\item[] \textbf{Operation ${\cal O}_1$.} Add a path $P_2$ to the tree $T_i$ where one vertex of the added path is colored amber and the other blue, and add an edge from the amber vertex in the added path $P_2$ to an amber vertex of $T_i$.
\item[] \textbf{Operation ${\cal O}_2$.} Add a path $P_4$ to the tree $T_i$ where the two leaves are colored blue and the two vertices of degree~$2$ are colored amber, and add an edge from a (blue) leaf of the added path $P_4$ to a blue vertex of $T_i$.
\end{itemize}
\end{definition}

To illustrate the definition, the trees $T_1, T_2, T_3$ shown in Figure~\ref{fig-tree} belong to the family~${\cal T}$, where $T_1$ is the $2$-colored base tree in ${\cal T}$, the tree $T_2$ is obtained from $T_1$ by applying Operation~${\cal O}_1$, and the tree $T_3$ is obtained from $T_2$ by applying Operation~${\cal O}_2$.

\begin{figure}[htb]
\begin{center}
\begin{tikzpicture}[scale=.85,style=thick,x=.85cm,y=.85cm]
\def\vr{2.5pt}
\path (0,1) coordinate (a1);
\path (1,1) coordinate (a2);
\path (2,1) coordinate (a3);
\path (3,1) coordinate (a4);
\path (6,1) coordinate (b1);
\path (7,1) coordinate (b2);
\path (8,1) coordinate (b3);
\path (9,1) coordinate (b4);
\path (7,0) coordinate (b5);
\path (8,0) coordinate (b6);
\path (12,1) coordinate (c1);
\path (13,1) coordinate (c2);
\path (14,1) coordinate (c3);
\path (15,1) coordinate (c4);
\path (13,0) coordinate (c5);
\path (14,0) coordinate (c6);
\path (16,1) coordinate (c7);
\path (17,1) coordinate (c8);
\path (17,0) coordinate (c9);
\path (16,0) coordinate (c10);
\draw (a1)--(a2)--(a3)--(a4);
\draw (b1)--(b2)--(b3)--(b4);
\draw (b6)--(b5)--(b2);
\draw (c1)--(c2)--(c3)--(c4);
\draw (c6)--(c5)--(c2);
\draw (c10)--(c9)--(c8)--(c7)--(c4);
\draw (a1) [fill=white] circle (\vr);
\draw (a2) [fill=white] circle (\vr);
\draw (a3) [fill=white] circle (\vr);
\draw (a4) [fill=white] circle (\vr);
\draw (b1) [fill=white] circle (\vr);
\draw (b2) [fill=white] circle (\vr);
\draw (b3) [fill=white] circle (\vr);
\draw (b4) [fill=white] circle (\vr);
\draw (b5) [fill=white] circle (\vr);
\draw (b6) [fill=white] circle (\vr);
\draw (c1) [fill=white] circle (\vr);
\draw (c2) [fill=white] circle (\vr);
\draw (c3) [fill=white] circle (\vr);
\draw (c4) [fill=white] circle (\vr);
\draw (c5) [fill=white] circle (\vr);
\draw (c6) [fill=white] circle (\vr);
\draw (c7) [fill=white] circle (\vr);
\draw (c8) [fill=white] circle (\vr);
\draw (c9) [fill=white] circle (\vr);
\draw (c10) [fill=white] circle (\vr);
\draw[anchor = south] (a1) node {{\small $B$}};
\draw[anchor = south] (a2) node {{\small $A$}};
\draw[anchor = south] (a3) node {{\small $A$}};
\draw[anchor = south] (a4) node {{\small $B$}};
\draw[anchor = south] (b1) node {{\small $B$}};
\draw[anchor = south] (b2) node {{\small $A$}};
\draw[anchor = south] (b3) node {{\small $A$}};
\draw[anchor = south] (b4) node {{\small $B$}};
\draw[anchor = north] (b5) node {{\small $A$}};
\draw[anchor = north] (b6) node {{\small $B$}};
\draw[anchor = south] (c1) node {{\small $B$}};
\draw[anchor = south] (c2) node {{\small $A$}};
\draw[anchor = south] (c3) node {{\small $A$}};
\draw[anchor = south] (c4) node {{\small $B$}};
\draw[anchor = north] (c5) node {{\small $A$}};
\draw[anchor = north] (c6) node {{\small $B$}};
\draw[anchor = south] (c7) node {{\small $B$}};
\draw[anchor = south] (c8) node {{\small $A$}};
\draw[anchor = north] (c9) node {{\small $A$}};
\draw[anchor = north] (c10) node {{\small $B$}};
\draw (1.5,-1.2) node {{\small (a) $T_1$}};
\draw (7.5,-1.2) node {{\small (b) $T_2$}};
\draw (15,-1.2) node {{\small (c) $T_3$}};
\end{tikzpicture}
\end{center}
\vskip -0.5 cm \caption{Trees $T_1, T_2, T_3$ in the family~${\cal T}$}
\label{fig-tree}
\end{figure}
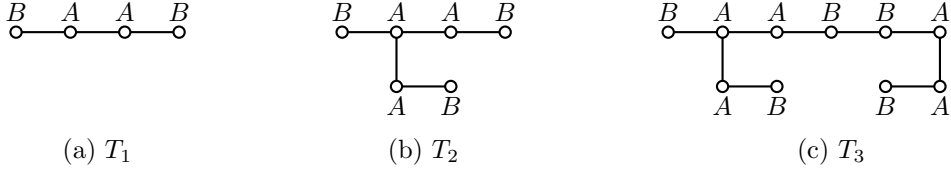

We shall need the following property of trees in the family~${\cal T}$.

\begin{observation}
\label{obser-treeF}
If $T$ is a tree in the family~${\cal T}$ of order~$n$, and $A$ and $B$ are the set of vertices in $T$ colored amber and blue, respectively, then $n \ge 4$ and the following properties hold. \\ [-22pt]
\begin{enumerate}
\item[{\rm (a)}] $|A| = |B| = \frac{1}{2}n$. \1
\item[{\rm (b)}] Every leaf of $T$ belongs to the set $B$.  \1
\item[{\rm (c)}] The set of edges between $A$ and $B$ induces a perfect matching in $T$, that is, every vertex in $A$ has exactly one neighbor in $B$ and every vertex in $B$ has exactly one neighbor in $A$. \1
\item[{\rm (d)}] The set $B$ is an ESD-set of $T$ of cardinality~$\frac{1}{2}n$. \1
\item[{\rm (e)}] The set $B$ is a $\gamma_{\esd}$-set of $T$.
\end{enumerate}
\end{observation}

\begin{theorem}
\label{tree-char}
If $T$ is a tree of order~$n \ge 2$, then $\gamma_{\esd}(T) = \frac{1}{2}n$ if and only if there exists a $2$-coloring of $T$ such that $T \in {\cal T}$.
\end{theorem}
\begin{proof}
If $T \in {\cal T}$, then by Observation~\ref{obser-treeF}, $\gamma_{\esd}(T) = \frac{1}{2}n$. To prove the necessity, let $T$ be a tree of (even) order~$n$ with $\gamma_{\esd}(T) = \frac{1}{2}n$. We proceed by induction on $n \ge 2$ to show that $T \in {\cal T}$. If $n \in \{2,3\}$, then $\gamma_{\esd}(T) = 2 > \frac{1}{2}n$, a contradiction. If $n = 4$ and $T$ is a star $K_{1,3}$, then $\gamma_{\esd}(T) = 3 > \frac{1}{2}n$, a contradiction. Hence if $n = 4$, then $T = P_4$ and $T$ can be $2$-colored so that $T \in {\cal T}$. This establishes the base case. Let $n \ge 6$ be an even integer, and let $T$ be a tree of order~$n$ with $\gamma_{\esd}(T) = \frac{1}{2}n$.

Let $B$ be a $\gamma_{\esd}$-set of $T$, and so $|B| = \frac{1}{2}n$. We note that every leaf of $T$ belongs to the set $B$. Let $P \colon v_0v_1 \ldots v_d$ be a longest path in $T$. Necessarily, $d$ is the diameter of $T$, that is, $d = \diam(T)$. Further, both $v_0$ and $v_d$ are leaves in $T$ and therefore belong to the ESD-set $B$. Since $n \ge 6$ and $T$ is not a star, we note that $d \ge 3$. Let $A = V \setminus B$, and so $|A| = |B| = \frac{1}{2}n$. Every vertex in $A$ is end super dominated by a unique vertex of $B$ and the set of edges between $A$ and $B$ induces a perfect matching in $T$. Hence, every vertex in $B$ has exactly one neighbor in $A$ and every vertex in $A$ has exactly one neighbor in $B$. We now consider the $2$-coloring of the vertices of $T$ where every vertex in $B$ is colored blue and every vertex in $A$ is colored amber.

Suppose that $\deg_T(v_1) \ge 3$. In this case, let $u_0$ be a child of $v_1$ different from $v_0$. By the maximality of the path $P$, the vertex $u_0$ is a leaf. Since $v_0$ is the only vertex in $B$ that is adjacent to~$v_1$, the vertex $u_0 \in A$. But then $u_0$ has no neighbor in $B$, a contradiction. Hence, $\deg_T(v_1) = 2$.

Suppose that $\deg_T(v_2) \ge 3$. In this case, let $T'$ be the tree $T - \{v_0,v_1\}$ of order~$n' = n - 2$, and let $A'$ and $B'$ be the restriction of $A$ and $B$ to the tree $T'$, that is, $A' = V(T) \cap A$ and $B' = V(T) \cap B$. The set $B'$ is an $\gamma_{\esd}$-set of $T'$, and so $\gamma_{\esd}(T') = \frac{1}{2}n - 1 = \frac{1}{2}n'$. Applying the inductive hypothesis to the tree $T'$, we have $T' \in {\cal T}$. Thus, $T'$ can be obtained from a sequence $T_1,\ldots,T_k$ of trees for some $k \ge 1$ by applying operations~${\cal O}_1$ and~${\cal O}_2$. In this case, the tree $T$ can be obtained from the tree $T'$ by applying operation~${\cal O}_1$. Hence, $T \in {\cal T}$, as desired.

Hence, we may assume that $\deg_T(v_2) = 2$, for otherwise the desired result holds. By our properties of the sets $A$ and $B$, we note that $\{v_0,v_3\} \in B$ and $\{v_1,v_2\} \in A$. Further, the vertices $v_1$ and $v_2$ are end super dominated by the vertices $v_0$ and $v_3$, respectively. Suppose that $\deg_T(v_3) \ge 3$. In this case, let $u_2$ be a child of $v_3$ different from $v_2$. Since $v_2$ is the only neighbor of $v_3$ that belongs to the set $A$, we note that $u_2 \in B$. Let $u_1$ be the vertex in $A$ that is end super dominated by the vertex $u_2$. We note that $u_1$ is a child of $u_2$. Since every vertex in $A$ has degree at least~$2$ in $T$, the vertex $u_1$ has degree at least~$2$. Let $u_0$ be a child of $u_1$. By the maximality of the path $P$, the vertex $u_0$ is a leaf. Since $u_2$ is the only vertex in $B$ that is adjacent to $u_1$, the vertex $u_0 \in A$. But then $u_0$ has no neighbor in $B$, a contradiction. Hence, $\deg_T(v_3) = 2$.

We now let $T'$ be the tree $T - \{v_0,v_1,v_2,v_3\}$ of order~$n' = n - 4$, and let $A'$ and $B'$ be the restriction of $A$ and $B$ to the tree $T'$, that is, $A' = V(T) \cap A$ and $B' = V(T) \cap B$. The set $B'$ is an $\gamma_{\esd}$-set of $T'$, and so $\gamma_{\esd}(T') = \frac{1}{2}n - 2 = \frac{1}{2}n'$. Applying the inductive hypothesis to the tree $T'$, we have $T' \in {\cal T}$. Thus, $T'$ can be obtained from a sequence $T_1,\ldots,T_k$ of trees for some $k \ge 1$ by applying operations~${\cal O}_1$ and~${\cal O}_2$. In this case, the tree $T$ can be obtained from the tree $T'$ by applying operation~${\cal O}_2$. Hence, $T \in {\cal T}$, as desired.~\qed
\end{proof}

\section{End super domination and size}

In 1965 Vizing~\cite{Vi-65} proved a classical result bounding the size of a graph in terms of its order and domination number. In this section we bound the size of a graph in terms of its order and end super domination number, and characterize the extremal graphs achieving equality in this bound.

Two edges in a graph $G$ are \emph{independent} if they are not adjacent in $G$; that is, if they have no vertex in common. A set of pairwise independent edges of $G$ is a \emph{matching} in $G$, while a matching of maximum cardinality is a \emph{maximum matching}. The number of edges in a maximum matching of $G$ is the \emph{matching number} of $G$ which we denote by $\alpha'(G)$. If $M$ is a matching of $G$, a vertex is \emph{$M$-matched} if it is incident with an edge of $M$; otherwise, the vertex is \emph{$M$-unmatched}.

\begin{theorem}
\label{esdom-size}
If $G$ is a graph of order~$n$ and size~$m$ with end super domination number $\gamma_{\esd}(G) = \gamma_{\esd}$, then
\[
m \le \binom{n}{2} - (n-\gamma_{\esd})(n-\gamma_{\esd}-1). \1
\]
Moreover, equality holds if and only if $G$ is obtained from a complete graph $K_n$ by removing all edges from a complete bipartite subgraph $K_{n-\gamma_{\esd},n-\gamma_{\esd}}$ of $G$, except for the edges of a perfect matching between the partite sets of the bipartite subgraph.
\end{theorem}
\begin{proof}
Let $\gamma_{\esd}(G) = \gamma_{\esd}$, and let $S$ be a $\gamma_{\esd}$-set of $G$. Thus, $S$ is an ESD-set of $G$ of minimum cardinality~$\gamma_{\esd}(G)$. We note that $\frac{1}{2}n \le \gamma_{\esd}$, or, equivalently, $0 \le 2\gamma_{\esd} - n$. Let $\overline{S} = \{u_1, \ldots, u_{n - \gamma_{\esd}}\}$ and let $w_i$ be a vertex in $S$ that end super dominates the vertex $u_i$ for $i \in [n - \gamma_{\esd}]$. Hence, $w_i \in S$ and $N(w_i) \cap \overline{S} = \{u_i\}$ for $i \in [n - \gamma_{\esd}]$. Let
\[
M = \bigcup_{i=1}^{n - \gamma_{\esd}} \{u_iw_i\}.
\]

Thus, $M$ is a matching in $G$ of cardinality~$n - \gamma_{\esd}$. We note that $N_G(u_i) \cap W = \{w_i\}$ and $N_G(w_i) \cap U = \{u_i\}$ for all $i \in [n - \gamma_{\esd}]$. Hence, $u_iw_j \notin E(G)$ for all $i$ and $j$ where $1 \le i < j \le n - \gamma_{\esd}$. Let $U = \{u_1,\ldots,u_{n - \gamma_{\esd}}\}$ and let $W = \{w_1,\ldots,w_{n - \gamma_{\esd}}\}$, and so $|U| = |W| = n - \gamma_{\esd}$. Each vertex in $U$ is not adjacent in $G$ to $|W| - 1$ vertices in $W$, implying that
\[
m \le \binom{n}{2} - |U| \times (|W| - 1)
= \binom{n}{2} - (n-\gamma_{\esd})(n-\gamma_{\esd} - 1). \1
\]
This establishes the desired upper bound. Moreover, if $m = \binom{n}{2} - (n-\gamma_{\esd})(n-\gamma_{\esd}-1)$, then the only edges missing in $G$ are edges of the form $u_iw_j$ where $i$ and $j$ and $1 \le i < j \le n - \gamma_{\esd}$. Thus if $F = \{u_iw_j \colon 1 \le i < j \le n - \gamma_{\esd}\}$, then $G = K_n - F$, that is, $G$ is obtained from $K_n$ by deleting all edges in $F$. Conversely, if $G$ is obtained from a complete graph $K_n$ by removing all edges from a complete bipartite subgraph $K_{n-\gamma_{\esd},n-\gamma_{\esd}}$ of $G$, except for the edges of a perfect matching between the partite sets of the bipartite subgraph, then $\gamma_{\esd}(G) = \gamma_{\esd}$ and $m = \binom{n}{2} - (n-\gamma_{\esd})(n-\gamma_{\esd}-1)$.~\qed
\end{proof}

We establish next a lower bound on the size of a connected graph in terms of its order and end super domination number.

\begin{theorem}
\label{edge-lower1}
If $G$ is a connected graph of order~$n$ and size~$m$ with end super domination number $\gamma_{\esd}(G) = \gamma_{\esd}$, then
\[
m \ge 2(n-\gamma_{\esd}) - 1, \1
\]
with equality if and only if $G$ is a tree and $G \in {\cal T}$.
\end{theorem}
\begin{proof}
Let $\gamma_{\esd}(G) = \gamma_{\esd}$. By Remark~\ref{thm-mix}, we have $\gamma_{\esd} \ge \frac{1}{2}n$. Since $G$ is a connected graph of order~$n$ and size~$m$, we have $m \ge n-1$. Hence,
\[
m \ge n - 1 = 2\left(n - \frac{1}{2}n\right) - 1 \ge 2(n-\gamma_{\esd}) - 1.
\]

Furthermore, suppose $m = 2(n-\gamma_{\esd}) - 1$. Thus the above inequalities are all equalities, implying that $m = n-1$ and $\gamma_{\esd} = \frac{1}{2}n$. In particular, $G$ is a tree of order~$n \ge 4$ with $\gamma_{\esd}(G) = \frac{1}{2}n$. Hence, by Theorem~\ref{tree-char}, we have $G \in {\cal T}$. Conversely, if $G \in {\cal T}$, then $\gamma_{\esd}(G) = \frac{1}{2}n$ and $m = n-1 = 2(n-\gamma_{\esd}) - 1$.
~\qed
\end{proof}

\section{End super domination number of $G-e$, $G/e$ and $G-v$}

Recall that if $e$ is an edge in a graph $G$, then $G-e$ is the graph obtained from $G$ by removing the edge $e$. In a graph $G$, contraction of an edge $e = uv$ with ends $u$ and $v$ is the replacement of $u$ and $v$ with a single vertex such that edges incident to the new vertex are the edges other than $e$ that were incident with $u$ or $v$. The resulting graph $G/e$ has one less edge than $G$ (\cite{Bondy}).  We refer the reader for more results about $G/e$ to \cite{Nima2}.
In this section, we  examine the effects on $\gamma_{\esd}(G)$ when $G$ is modified by an edge removal and edge contraction.

First, we establish upper and lower bounds for the end super domination number of a graph obtained from edge removal.

\begin{theorem}
\label{G-e-esp}
If $G=(V,E)$ is a graph and $e=uv \in E$, then
\[
\gamma_{\esd}(G)-1 \le \gamma_{\esd}(G-e)\le \gamma_{\esd}(G)+2.
\]
\end{theorem}
\begin{proof}
First we find the upper bound for $\gamma_{\sd}(G-e)$. Suppose that $S$ is a $\gamma_{\esd}$-set of $G$, and so $|S| = \gamma_{\esd}(G)$. Since $S$ is an EPD-set of $G$, removing the edge $e$ from $G$ and letting $S'=S\cup \{u,v\}$, yields an EPD-set of $G$', and so $\gamma_{\esd}(G-e) \le |S'| \le |S| + 2 = \gamma_{\esd}(G)+2$.

Next we determine a lower bound for $\gamma_{\sd}(G-e)$. Suppose that $S$ is a $\gamma_{\esd}$-set of $G-e$, and so $|S| = \gamma_{\esd}(G-e)$. If $\{u,v\} \subseteq S$, then in this case, $S$ is also an ESD-set of $G$. If $S \cap \{u,v\} = \emptyset$, then once again $S$ is also an ESD-set of $G$. In both cases, $\gamma_{\esd}(G) \le |S| = \gamma_{\esd}(G-e)$. Hence we may assume that $|S \cap \{u,v\}| = 1$. Renaming vertices if necessary, we may assume that $u \in S$ and $v \notin S$. In this case, every vertex $x \in \overline{S}$ is end super dominated in $G-e$ by a vertex $x'$ of $S$, that is, $N_{G-e}(x') \cap \overline{S} = \{x\}$. Adding back the edge $e$, the set $S \cup \{v\}$ is an end super dominating set of $G$, noting that every vertex $x \in \overline{S} \setminus \{v\}$ is end super dominated in $G$ by the vertex $x'$ defined earlier. Thus, $\gamma_{\esd}(G) \le |S \cup \{v\}| = |S| + 1 = \gamma_{\esd}(G-e)$.~\qed
\end{proof}
	
We remark that if $G=(V,E)$ is a graph and $e=uv \in E$ where $\deg_G(u) \ge 3$, then the upper bound in Theorem~\ref{G-e-esp} can be improved slightly. In this case, if $S$ is a $\gamma_{\esd}$-set of $G$, then the set $S \cup \{v\}$ is an ESD-set of $G-e$, and so $\gamma_{\esd}(G-e) \le |S|+ 1 = \gamma_{\esd}(G) + 1$.

The bounds in Theorem~\ref{G-e-esp} are sharp. To see that the upper bound is sharp, for $k \ge 1$ let $G$ be the path $P_{4k}$ given by $v_1v_2 \ldots v_{4k}$ and let $e = v_2v_3$. Thus, $G - e = P_2 \cup P_{4k-2}$, and so Theorem~\ref{thm-path-esp} we have $\gamma_{\esd}(G) = 2k$ and $\gamma_{\esd}(G-e) = \gamma_{\esd}(P_2) + \gamma_{\esd}(P_{4k-2}) = 2 + 2k = \gamma_{\esd}(G) + 2$. To see that the lower bound in sharp, the simplest example is to take $G = K_4 - f$ where $f$ is an arbitrary edge of the $K_4$, and let $e$ be the edge in $G$ that joins the two vertices of degree~$3$ in $G$, and so $G - e = C_4$. In this case, $\gamma_{\esd}(G) = 3$ and $\gamma_{\esd}(G-e) = 2 = \gamma_{\esd}(G) - 1$.

In our earlier example when $G = P_{4k}$ where $k \ge 1$, we note that the removal of the selected edge $e$ from $G$ produces a disconnected graph. This raises the question of whether the upper bound in Theorem~\ref{G-e-esp} can be improved if $G-e$ is connected. The following example shows that it is not possible.

\begin{example}
{\normalfont
Consider the graph $G$ shown in Figure~\ref{G-e-upper-esp-connected}, and let $e=v_{10}v_{16}$. The graph $G$ has order~$n = 20$. Further, the set $S=\{v_1,v_4,v_5,v_8,v_9,v_{11},v_{14},v_{15},v_{17},v_{20}\}$ is an ESD-set of $G$ of cardinality~$\frac{1}{2}n = 10$. By Remark \ref{thm-mix}, the set $S$ is therefore a $\gamma_{\esd}$-set of $G$, and so $\gamma_{\esd}(G) = 10$. However it is simple exercise to check (or use a computer) that $\gamma_{\esd}(G - e) = 12 = \gamma_{\esd}(G)$. The set $S \cup \{v_{10},v_{16}\}$ is an example of a $\gamma_{\esd}$-set of $G - e$.
}
\end{example}

\begin{figure}
		\begin{center}
			\psscalebox{0.6 0.6}
{
\begin{pspicture}(0,-6.015)(8.72,-0.045)
\psdots[linecolor=black, dotsize=0.4](0.28,-2.215)
\psdots[linecolor=black, dotsize=0.4](1.88,-2.215)
\psdots[linecolor=black, dotsize=0.4](3.48,-2.215)
\psdots[linecolor=black, dotsize=0.4](5.08,-2.215)
\psdots[linecolor=black, dotsize=0.4](6.68,-2.215)
\psdots[linecolor=black, dotsize=0.4](8.28,-2.215)
\psdots[linecolor=black, dotsize=0.4](0.28,-3.815)
\psdots[linecolor=black, dotsize=0.4](1.88,-3.815)
\psdots[linecolor=black, dotsize=0.4](3.48,-3.815)
\psdots[linecolor=black, dotsize=0.4](5.08,-3.815)
\psdots[linecolor=black, dotsize=0.4](6.68,-3.815)
\psdots[linecolor=black, dotsize=0.4](8.28,-3.815)
\psdots[linecolor=black, dotsize=0.4](1.88,-5.415)
\psdots[linecolor=black, dotsize=0.4](0.28,-5.415)
\psdots[linecolor=black, dotsize=0.4](3.48,-5.415)
\psdots[linecolor=black, dotsize=0.4](5.08,-5.415)
\psdots[linecolor=black, dotsize=0.4](0.28,-0.615)
\psdots[linecolor=black, dotsize=0.4](1.88,-0.615)
\psdots[linecolor=black, dotsize=0.4](3.48,-0.615)
\psdots[linecolor=black, dotsize=0.4](5.08,-0.615)
\psline[linecolor=black, linewidth=0.08](0.28,-2.215)(8.28,-2.215)(8.28,-3.815)(0.28,-3.815)(0.28,-2.215)(0.28,-2.215)
\psline[linecolor=black, linewidth=0.08](0.28,-0.615)(1.88,-0.615)(1.88,-5.415)(0.28,-5.415)(0.28,-5.415)
\psline[linecolor=black, linewidth=0.08](5.08,-0.615)(3.48,-0.615)(3.48,-5.415)(5.08,-5.415)(5.08,-5.415)
\psline[linecolor=black, linewidth=0.08](5.08,-2.215)(5.08,-3.815)(5.08,-3.815)
\psline[linecolor=black, linewidth=0.08](6.68,-2.215)(6.68,-3.815)(6.68,-3.815)
\rput[bl](0.06,-0.295){$v_1$}
\rput[bl](1.74,-0.355){$v_2$}
\rput[bl](3.34,-0.395){$v_3$}
\rput[bl](4.9,-0.395){$v_4$}
\rput[bl](0.08,-1.955){$v_5$}
\rput[bl](1.96,-1.955){$v_6$}
\rput[bl](3.56,-1.975){$v_7$}
\rput[bl](4.9,-1.975){$v_8$}
\rput[bl](6.52,-1.955){$v_9$}
\rput[bl](8.08,-1.995){$v_{10}$}
\rput[bl](0.1,-4.335){$v_{11}$}
\rput[bl](2.0,-4.295){$v_{12}$}
\rput[bl](3.64,-4.355){$v_{13}$}
\rput[bl](4.78,-4.355){$v_{14}$}
\rput[bl](6.42,-4.335){$v_{15}$}
\rput[bl](8.06,-4.295){$v_{16}$}
\rput[bl](8.52,-3.115){$e$}
\rput[bl](0.0,-5.995){$v_{17}$}
\rput[bl](1.6,-5.995){$v_{18}$}
\rput[bl](3.2,-6.015){$v_{19}$}
\rput[bl](4.82,-5.975){$v_{20}$}
\end{pspicture}
}
		\end{center}
		\caption{Graph $G$} \label{G-e-upper-esp-connected}
	\end{figure}

In \cite{Nima}, it is shown that if $G=(V,E)$ is a graph and $e\in E$, then for $G/e$ we have $\gamma_{\sd}(G) - 1 \le \gamma_{\sd}(G/e)\le \gamma_{\sd}(G)$. Moreover, it is shown that these bounds are sharp. Since edge contraction does not create additional vertices of degree~$1$, an identical proof as in~\cite{Nima} yields the following result.

\begin{proposition}
\label{G/e-esp}
If $G=(V,E)$ is a graph and $e\in E$, then
\[
\gamma_{\esd}(G) - 1 \le \gamma_{\esd}(G/e)\le \gamma_{\esd}(G).
\]
\end{proposition}

Recall that if $v$ is a vertex in a graph $G$, then $G-v$ is the graph obtained from $G$ by deleting the edge~$v$ (and all edges incident with $v$). We next examine the effects on $\gamma_{\esd}(G)$ when $G$ is modified by a vertex removal.

\begin{proposition}
\label{G-v-esp}
If $G=(V,E)$ is a graph and $v \in V$, then
\[
\gamma_{\esd}(G)-1 \le \gamma_{\esd}(G-v) \le \gamma_{\esd}(G)+\deg_G(v)-1.
\]
\end{proposition}
\begin{proof}
Let $S$ be a $\gamma_{\esd}$-set of $G$. If $v \notin S$, then $S$ is an ESD-set for $G-v$, and in this case $\gamma_{\esd}(G-v) \le |S| = \gamma_{\esd}(G)$. Hence we may assume that $v \in S$. The set $(S \setminus \{v\}) \cup N(v)$ is an ESD-set for $G-v$, and so in this case $\gamma_{\esd}(G-v) \le (|S| - 1) + \deg_G(v) = \gamma_{\esd}(G) + \deg(v)-1$. This establishes the desired upper bound. To prove the lower bound for  $\gamma_{\esd}(G-v)$, we note that every $\gamma_{\esd}$-set of $G-v$ can be extended to an ESP-set of $G$ by adding to it the vertex~$v$, implying that $\gamma_{\esd}(G) \le \gamma_{\esd}(G-v) + 1$.~\qed
\end{proof}			

We remark that the bounds in Theorem \ref{G-v-esp} are sharp.  To see that the upper bound is sharp, for $k \ge 2$ let $G$ be graph obtained from a star $K_{1,k}$ with central vertex~$v$ by subdividing every edge exactly once. The resulting graph $G$ satisfies $\gamma_{\esd}(G) = k+1$, where the set of $k$ leaves, together with the vertex~$v$, form a $\gamma_{\esd}$-set of $G$. However the graph $G-v$ consists of $k$ vertex disjoint copies of $K_2$, and so $\gamma_{\esd}(G-v) = 2k = (k+1) + k - 1 = \gamma_{\esd}(G)+\deg_G(v)-1$. To see that the lower bound in sharp, for $k \ge 1$ let $G$ be the path $P_{4k+1}$ and let $v$ be a leaf of $G$. We note that $G - v = P_{4k}$. By Theorem~\ref{thm-path-esp}, $\gamma_{\esd}(G) = 2k+1$ and $\gamma_{\esd}(G-v) = 2k = \gamma_{\esd}(G) - 1$.

\section{A linear approach to end super domination number}

In this section, we consider the rank of the adjacency matrix of a graph and find a tight lower bound for it, based on its end super domination  number.

\begin{theorem}
If $G$ is a connected graph of order $n$ with adjacency matrix $A$, then
\[
\rank(A) \ge n-\gamma_{\esd}(G),
\]
with equality if and only if $G$ is a complete bipartite graph $K_{n,m}$ where $\min\{n,m\}\ge 2$.
\end{theorem}
\begin{proof}
For simplicity let $\gamma_{\esd} = \gamma_{\esd}(G)$. Let $V(G)=\{v_1,v_2,\ldots,v_n\}$. There exists a  matching $M$ of size $n-\gamma_{\esd}$, say
\[
M = \bigcup_{i=1}^{n-\gamma_{\esd}} \{v_i v_{n-\gamma_{\esd}+i} \}
\]
where $\{v_1,v_2,\ldots,v_{n-\gamma_{\esd}}\}\subseteq S$, and $S$ is a $\gamma_{\esd}$-set of $G$, and moreover $v_iv_j \notin E(G)$, for each $i,j$, $1\le i\le n-\gamma_{\esd}$ and  $n-\gamma_{\esd} + 1\le j\le 2(n-\gamma_{\esd})$. Let
\[
L_1 = \bigcup_{i=1}^{n-\gamma_{\esd}} \{v_i\}
\hspace*{0.25cm} \mbox{and} \hspace*{0.25cm}
L_2 = \bigcup_{i=1}^{n-\gamma_{\esd}} \{v_{n-\gamma_{\esd}+i} \}
\]
and let $L_3=V(G)\setminus(L_1\cup L_2)$. We now consider the adjacency matrix $A$ as follows:

\[
A=\begin{blockarray}{cccc}
 & L_1 & L_2 & L_3 \\
\begin{block}{c[c|c|c]}
L_1 & B & I_{n-\gamma_{\esd}} & D^T \\
\cmidrule{2-4}
L_2 & I_{n-\gamma_{\esd}} & C & E^T \\
\cmidrule{2-4}
L_3 & D & E & F\\
\end{block}
\end{blockarray}
\]

Since the rows of $A$ corresponding to $L_1$ are linearly independent, we have $\rank(A)\ge n-\gamma_{\esd}(G)$. Now, assume that $\rank(A)= n-\gamma_{\esd}(G)$. In this case, every row of $A$ corresponding to $L_2$ is a linear combination of rows in $L_1$. Since $C$ is a $(0,1)$-matrix, every coefficient in the linear combination is $0$ or $1$. In the $L_2\times L_1$ matrix, we have the identity matrix  $I_{n-\gamma_{\esd}}$ and every row of $I_{n-\gamma_{\esd}}$ is a linear combination of rows of $B$, implying that all columns of $B$ are non-zero. On the other hand, no row of $B$ has more than one component equal to $1$. Thus since $B$ is a symmetric matrix, we conclude that $B$ is a permutation matrix. This implies that every row of $A$ corresponding to $L_2$ is exactly one row of $A$ corresponding to $L_1$. Therefore, $B=C$.

Since $B$ is a permutation matrix, the induced vertex subgraphs $G[L_1]$ and $G[L_2]$ are $1$-regular graphs. Now, since $B=C$, without loss of generality, we may assume that  $G[L_1 \cup L_2]$ is a disjoint union of $\frac{1}{2}(n-\gamma_{\esd})$ cycles of length~$4$. Since  every row of $A$ corresponding to $L_3$ is a linear combination of rows of $A$ corresponding to $L_2$ and since $C=B$ is a permutation matrix and $E$ is a $(0,1)$-matrix, we conclude that every row of $A$ in $L_3$ is exactly a row of $A$ in $L_2$. This implies that every row of $D$ has exactly one component equal to~$1$. The same property holds for $E$. Thus, every vertex in $L_3$ is adjacent to exactly one vertex in $L_1$ and one vertex in $L_2$. We note that the rank of the adjacency matrix of $C_4$ is~$2$. Hence for every natural number $t$, the rank of adjacency matrix of the disjoint union of $t$ copies of $C_4$ is $2t$. Now, let $v_i\in L_3$. If $v_i$ is adjacent to two different copies of $C_4$ in $G[L_1 \cup L_2]$, then $G$ has an induced vertex subgraph as illustrated in Figure \ref{inducedG}. Therefore, $G$ contains $P_6$ as an induced vertex subgraph, implying that
\[
n-\gamma_{\esd}=\rank(A)\ge 6+2\left( \frac{n-\gamma_{\esd}}{2}-2 \right),
\]
a contradiction. Thus every $v_i \in L_3$ is adjacent to exactly two vertices of only one copy of $C_4$ in $G[L_1 \cup L_2]$.

\begin{figure}[!h]
		\begin{center}
			\psscalebox{0.6 0.6}
{
\begin{pspicture}(0,-4.6014423)(6.794231,-2.2043269)
\psdots[linecolor=black, dotsize=0.4](0.19711548,-2.4014423)
\psdots[linecolor=black, dotsize=0.4](1.7971154,-2.4014423)
\psdots[linecolor=black, dotsize=0.4](1.7971154,-4.0014424)
\psdots[linecolor=black, dotsize=0.4](0.19711548,-4.0014424)
\psdots[linecolor=black, dotsize=0.4](3.3971155,-4.0014424)
\psdots[linecolor=black, dotsize=0.4](4.9971156,-4.0014424)
\psdots[linecolor=black, dotsize=0.4](4.9971156,-2.4014423)
\psdots[linecolor=black, dotsize=0.4](6.5971155,-2.4014423)
\psdots[linecolor=black, dotsize=0.4](6.5971155,-4.0014424)
\psline[linecolor=black, linewidth=0.08](1.7971154,-4.0014424)(1.7971154,-2.4014423)(0.19711548,-2.4014423)(0.19711548,-4.0014424)(6.5971155,-4.0014424)(6.5971155,-2.4014423)(4.9971156,-2.4014423)(4.9971156,-4.0014424)(4.9971156,-4.0014424)
\rput[bl](3.2771156,-4.6014423){$v_i$}
\end{pspicture}
}
		\end{center}
		\caption{Induced vertex subgraph of $G$} \label{inducedG}
	\end{figure}

Suppose next that $G$ has an induced vertex subgraph as illustrated in Figure~\ref{inducedG2}. In this case, since the rank of the adjacency matrix of $C_3$ is~$3$, we conclude that
\[
n-\gamma_{\esd}=\rank(A)\ge 3+2\left( \frac{n-\gamma_{\esd}}{2}-1 \right),
\]
a contradiction.

\begin{figure}[!h]
		\begin{center}
			\psscalebox{0.6 0.6}
{
\begin{pspicture}(0,-4.4)(4.1271152,-2.405769)
\psline[linecolor=black, linewidth=0.08](0.19711548,-2.6028845)(1.7971154,-2.6028845)(1.7971154,-4.2028847)(0.19711548,-4.2028847)(0.19711548,-2.6028845)(0.19711548,-2.6028845)
\psline[linecolor=black, linewidth=0.08](1.7971154,-2.6028845)(3.3971155,-3.4028845)(1.7971154,-4.2028847)(1.7971154,-4.2028847)
\psdots[linecolor=black, dotsize=0.4](1.7971154,-2.6028845)
\psdots[linecolor=black, dotsize=0.4](0.19711548,-2.6028845)
\psdots[linecolor=black, dotsize=0.4](0.19711548,-4.2028847)
\psdots[linecolor=black, dotsize=0.4](1.7971154,-4.2028847)
\psdots[linecolor=black, dotsize=0.4](3.3971155,-3.4028845)
\rput[bl](3.7971156,-3.4028845){$v_i$}
\end{pspicture}
}
		\end{center}
		\caption{Induced vertex subgraph of $G$} \label{inducedG2}
	\end{figure}

Now, assume that $G[L_1 \cup L_2]$ contains at least two copies of $C_4$. Since $G$ is connected, there exist two vertices $v_i$ and $v_j$ in $L_3$ such that $G$ has an induced vertex subgraph as illustrated in Figure~\ref{inducedG3}, where $H$ is a connected subgraph of $G[L_3]$. Clearly, this graph contains $P_6$ or $3P_2$ as an induced vertex subgraph. Hence we infer that
\[
n-\gamma_{\esd}=\rank(A)\ge 6+2\left( \frac{n-\gamma_{\esd}}{2}-2 \right),
\]
a contradiction. Thus, $G[L_1 \cup L_2] = C_4$, and $\frac{1}{2}(n-\gamma_{\esd})=1$. Therefore, $\rank(A)=n-\gamma_{\esd}=2$, and by~\cite{Akbari}, $G$ is a complete bipartite graph.\qed
\end{proof}

\begin{figure}[!h]
		\begin{center}
			\psscalebox{0.6 0.6}
{
\begin{pspicture}(0,-4.8)(12.394231,-2.005769)
\rput[bl](4.7771153,-3.1628845){$v_i$}
\psdots[linecolor=black, dotsize=0.4](1.3971155,-2.2028844)
\psdots[linecolor=black, dotsize=0.4](0.19711548,-3.4028845)
\psdots[linecolor=black, dotsize=0.4](1.3971155,-4.6028843)
\psdots[linecolor=black, dotsize=0.4](2.5971155,-3.4028845)
\psdots[linecolor=black, dotsize=0.4](4.5971155,-3.4028845)
\psline[linecolor=black, linewidth=0.08](4.5971155,-3.4028845)(1.3971155,-2.2028844)(0.19711548,-3.4028845)(1.3971155,-4.6028843)(2.5971155,-3.4028845)(1.3971155,-2.2028844)(1.3971155,-2.2028844)
\psline[linecolor=black, linewidth=0.08](1.3971155,-4.6028843)(4.5971155,-3.4028845)(4.5971155,-3.4028845)
\psdots[linecolor=black, dotsize=0.4](7.7971153,-3.4028845)
\psdots[linecolor=black, dotsize=0.4](9.797115,-3.4028845)
\psdots[linecolor=black, dotsize=0.4](10.997115,-2.2028844)
\psdots[linecolor=black, dotsize=0.4](12.197116,-3.4028845)
\psdots[linecolor=black, dotsize=0.4](10.997115,-4.6028843)
\psline[linecolor=black, linewidth=0.08](7.7971153,-3.4028845)(10.997115,-2.2028844)(12.197116,-3.4028845)(10.997115,-4.6028843)(9.797115,-3.4028845)(10.997115,-2.2028844)(10.997115,-2.2028844)
\psline[linecolor=black, linewidth=0.08](7.7971153,-3.4028845)(10.997115,-4.6028843)(10.997115,-4.6028843)
\psellipse[linecolor=black, linewidth=0.08, linestyle=dashed, dash=0.17638889cm 0.10583334cm, dimen=outer](6.1971154,-3.4028845)(2.4,1.2)
\rput[bl](7.3171153,-3.1828845){$v_j$}
\rput[bl](5.8571153,-4.2028847){\large{$H$}}
\end{pspicture}
}
		\end{center}
		\caption{Induced vertex subgraph of $G$} \label{inducedG3}
	\end{figure}

\section{Enumeration of minimum end super dominating sets}

This section is devoted to enumeration of minimum end super dominating sets, which involves counting the number of distinct $\gamma_{\esd}$-sets of a graph $G$. Given a graph $G$, let ${\cal N}_{\esd}(G)$ be the family of $\gamma_{\esd}$-sets of $G$ and let $N_{\esd}(G)=|{\cal N}_{\esd}(G)|$. Our aim is to compute $N_{\esd}(G)$ for special graph classes. Building on the result of
Proposition~\ref{thm-special}, one can readily determine the value $N_{\esd}(G)$ when $G$ is a complete graph, a complete bipartite graph, or a star graph.

\begin{example}
{\normalfont
The following properties hold. \\ [-22pt]
\begin{enumerate}
\item[{\rm (a)}] For $n \ge 3$, $N_{\esd}(K_n)= n$. \1
\item[{\rm (b)}] For $\min\{n,m\} \ge 2$, $N_{\esd}(K_{n,m}) = nm$. \1
\item[{\rm (c)}] For $n \ge 2$, $N_{\esd}(K_{1,n})= 1$. \1
\end{enumerate}
}
\end{example}

By Remark \ref{thm-mix}, if $G$ has no pendant vertices, then $\gamma_{\esd}(G) = \gamma_{\sd}(G)$. Hence if $G$ has no pendant vertices, then $N_{\esd}(G)$ is equal to the number of $\gamma_{\sd}$-sets of $G$. In particular, if $G$ is a connected graph with minimum degree at least~$2$, then $N_{\esd}(G)$ is equal to the number of $\gamma_{\sd}$-sets of $G$. The number of such sets when $G$ is a cycle was enumerated in~\cite{Nima1}.  We present here a simplified proof of this result using a different counting approach. Parts of the proof will be needed when proving the result for paths.

\begin{theorem}\label{Cycle-counting}
If $C_n$ is the cycle graph of order $n\ge 3$,  then
\begin{displaymath}
N_{\esd}(C_n)= \left\{ \begin{array}{ll}
4 & \textrm{if $n =4k$, } \1 \\
2n & \textrm{if $n =4k+1$,} \1\\
\frac{5}{8}(n^2-2n) & \textrm{if $n =4k+2$,} \1 \\
n & \textrm{if $n =4k+3$.}
\end{array} \right.
\end{displaymath}
\end{theorem}
\begin{proof}
For $n \ge 3$, let $G$ be a cycle $C_n$ given by $v_1v_2 \ldots v_nv_1$, and let $S$ be a $\gamma_{\esd}$-set of $G$. As before, let $V = V(G)$. We consider four cases.

\medskip
\emph{Case~1: $n=4k$, for some $k \ge 1$.} By Proposition \ref{thm-special}, in this case we have $|S| = 2k = \frac{1}{2}n$. We show that for every three consecutive vertices $v_{i-1}$, $v_i$ and $v_{i+1}$ where addition is taken modulo~$n$, if $v_i \in S$, then at least one of $v_{i-1}$ and $v_{i+1}$ must belong to $S$.  Suppose, to the contrary, that $v_i \in S$ and $\{v_{i-1},v_{i+1}\} \subseteq \overline{S}$. If $n = 4$, then $S$ is not an ESD-set of $G$, a contradiction. Hence, $n \ge 8$. By definition of an ESD-set, we infer that $\{v_{i-3},v_{i-2},v_{i+2},v_{i+3}\} \subseteq S$. If $n = 8$, then $|S| \ge 5 > \frac{1}{2}n$, a contradiction. Hence, $n \ge 12$. Thus, letting $G' = G - \{v_{i-2},v_{i-1},v_i,v_{i+1},v_{i+2}\}$ and $S' = S \setminus \{v_{i-2},v_i,v_{i+2}\}$, we note that $G' = P_{4k-5}$ and $S'$ is an ESD-set of $G'$, implying by Theorem~\ref{thm-path-esp} that $2k-2 = \gamma_{\esd}(G') \le |S'| = |S| - 3$, and so $|S| = 2k+1$, a contradiction. Hence if $v_i \in S$, then $S \cap \{v_{i-1},v_{i+1}\}| \ge 1$.

Suppose next that $\{v_{i-1},v_i,v_{i+1}\} \subseteq S$. If $n = 4$, then $|S| = 3 > \frac{1}{2}n$, a contradiction. Hence, $n \ge 8$. We now let $G' = G - v_i$ and let $S' = S \setminus \{v\}$, we note that $G' = P_{4k-1}$ and $S'$ is an ESD-set of $G'$, implying by Theorem~\ref{thm-path-esp} that $2k = \gamma_{\esd}(G') \le |S'| = |S| - 1$, and so $|S| = 2k+1$, a contradiction Hence, if $v_i \in S$, then $S \cap \{v_{i-1},v_i,v_{i+1}\}| = 2$ for all $i \in [n]$ (where addition is taken modulo~$n$). We infer that the only $\gamma_{\esd}$-set of $G$ in this case when $n = 4k$ are the sets
\[
S_1 = \bigcup_{i=0}^{k-1} \{v_{4i+1},v_{4i+2}\}, \hspace*{0.5cm}
S_2 = \bigcup_{i=0}^{k-1} \{v_{4i+2},v_{4i+3}\},
\]
and
\[
S_3 = \bigcup_{i=0}^{k-1} \{v_{4i+3},v_{4i+4}\}, \hspace*{0.5cm}
S_4 = \bigcup_{i=0}^{k-1} \{v_{4i+1},v_{4i+4}\},
\]
whence $N_{\esd}(G)=4$.

\medskip
\emph{Case~2: $n=4k+1$, for some $k \ge 1$.} By Proposition \ref{thm-special}, in this case we have $|S| = 2k + 1$.We now consider two cases.

Suppose that $S$ contains a vertex with both its neighbors in $\overline{S}$.  These three consecutive vertices can be chosen in $n$ ways. Renaming vertices if necessary, we may assume that $v_2 \in S$ and $\{v_1,v_3\} \subseteq \overline{S}$. In this case, $\{v_4,v_5,v_{4k},v_{4k+1}\} \subseteq {S}$. Suppose that $v_6 \in S$. Letting $G' = G - \{v_3,v_4,v_5\}$ and $S' = S \setminus \{v_4,v_5\}$, we note that $G' = P_{4k-2}$ and $S'$ is an ESD-set of $G'$, implying by Theorem~\ref{thm-path-esp} that $2k = \gamma_{\esd}(G') \le |S'| = |S| - 2$, and so $|S| = 2k+2$, a contradiction. Hence, $v_6 \notin S$.

Suppose next that $v_7 \in S$. If $v_8 \in S$, then letting $G' = G - \{v_5,v_6,v_7\}$ and $S' = S \setminus \{v_5,v_7\}$, we note that $G' = P_{4k-2}$ and $S'$ is an ESD-set of $G'$, implying by Theorem~\ref{thm-path-esp} that $2k = \gamma_{\esd}(G') \le |S'| = |S| - 2$, and so $|S| = 2k+2$, a contradiction. Hence, $v_8 \notin S$. The vertex $v_8$ is end super dominated by the vertex $v_9 \in S$, implying that $v_{10} \in S$. Letting $G' = G - \{v_{4k+1},v_1,v_2,\ldots,v_9\}$ and $S' = S \setminus \{v_{4k+1},v_2,v_4,v_5,v_7,v_9\}$, we note that $G' = P_{4k-9}$ and $S'$ is an ESD-set of $G'$, implying by Theorem~\ref{thm-path-esp} that $2k - 4 = \gamma_{\esd}(G') \le |S'| = |S| - 6$, and so $|S| = 2k+2$, a contradiction. Hence, $v_7 \notin S$, implying that $\{v_8,v_9\} \subset S$. Continuing this process, we have
\[
S = \{v_2,v_4,v_{4k+1}\} \cup \left( \bigcup_{i=1}^{k-1} \{v_{4i+1},v_{4i+4}\} \right).
\]
Thus in this case when $v_2 \in S$ and $\{v_1,v_3\} \subseteq \overline{S}$, the set $S$ is uniquely determined.

Suppose secondly that $S$ contains three consecutive vertices on the cycle. These three consecutive vertices can be chosen in $n$ ways. We may assume that $\{v_1,v_2,v_3\} \subseteq S$. If $v_4 \in S$, then letting $G' = G - \{v_2,v_3\}$ and $S' = S \setminus \{v_2,v_3\}$, we note that $G' = P_{4k-1}$ and $S'$ is an ESD-set of $G'$, implying by Theorem~\ref{thm-path-esp} that $2k = \gamma_{\esd}(G') \le |S'| = |S| - 2$, and so $|S| = 2k+2$, a contradiction. Hence, $v_4 \notin S$. If $v_5 \in S$, then letting $G' = G - \{v_2,v_,v_4\}$ and $S' = S \setminus \{v_2,v_3,v_4\}$, we note that $G' = P_{4k-2}$ and $S'$ is an ESD-set of $G'$, implying by Theorem~\ref{thm-path-esp} that $2k = \gamma_{\esd}(G') \le |S'| = |S| - 2$, and so $|S| = 2k+2$, a contradiction. Hence, $v_5 \notin S$, implying that $\{v_6,v_7\} \subseteq S$. Continuing this process, we have
\[
S = \{v_1,v_{4k+1}\} \cup \left( \bigcup_{i=1}^{k-1} \{v_{4i+2},v_{4i+3}\} \right).
\]
Thus in this case when $\{v_1,v_2,v_3\} \subseteq S$, the set $S$ is uniquely determined. Thus, $N_{\esd}(G) = 2n$.

\medskip
\emph{Case~3: $n=4k+2$, for some $k \ge 1$.}  By Proposition \ref{thm-special},  in this case we have  $|S|=2k+2$. Since choosing any four vertices from $C_6$ gives us an $\gamma_{\esd}$-set of $G$, the formula holds for $n=6$. Hence we may assume that $n \ge 10$. Suppose that $S$ contains five consecutive vertices on the cycle. Renaming vertices if necessary, we may assume that $\{v_1,v_2,v_3,v_4,v_5\} \subseteq S$. Letting $G' = G - \{v_2,v_3,v_4\}$ and $S' = S \setminus \{v_2,v_3,v_4\}$, we note that $G' = P_{4k-1}$ and $S'$ is an ESD-set of $G'$, implying by Theorem~\ref{thm-path-esp} that $2k = \gamma_{\esd}(G') \le |S'| = |S| - 3$, and so $|S| = 2k+3$, a contradiction. Hence, $S$ cannot contain five consecutive vertices on the cycle.
So it remains to consider cases with only one subset of consecutive vertices of $S$ with cardinality~$4$, two subsets of consecutive vertices of $S$ with cardinality~$3$, one subset of consecutive vertices of $S$ with cardinality~$3$ and all subsets of consecutive vertices of $S$ with cardinality less than~$3$. We have two subcases:

\begin{figure}[!h]
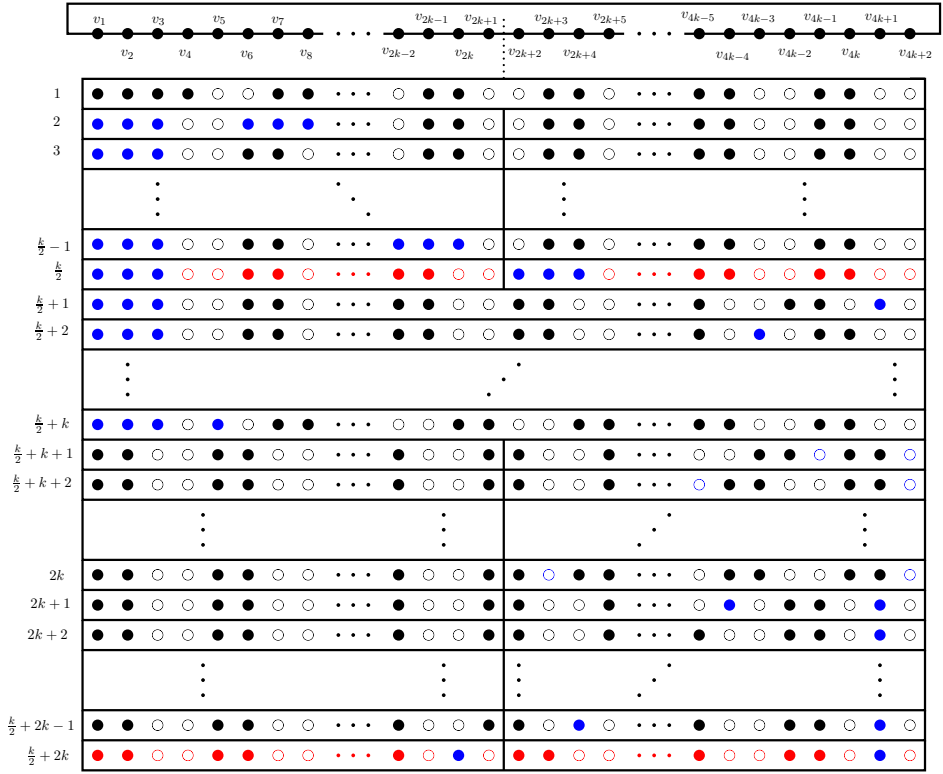

		\begin{center}
			\psscalebox{0.5 0.5}
{

}
		\end{center}
		\caption{Making the ESD-sets of cycle graph of order  $n \equiv 2 ~ (mod ~ 8)$ } \label{cycle2mod8}
	\end{figure}

\medskip
\emph{Case~3.1. $n \equiv 2 \, (\modo \, 8)$.} In this case, $k$ is even. We consider the case when $\{v_1,v_2\} \subseteq S$ as shown in Figure \ref{cycle2mod8}. One can easily check that if we shift the indices, then we have all cases of $\gamma_{\esd}$-set of $G$ and we have no more cases (consider movements of blue vertices in Figure~\ref{cycle2mod8}, where filled ones are in $S$ and empty ones are in $\overline{S}$).  All cases can be shifted $4k+2$ times except the rows $\frac{k}{2}$ and $\frac{k}{2}+2k$ which can be shifted $2k+1$ times. So in general we have $\frac{5k}{2}-2$ cases which can be shifted $n$ times  and two cases which can be shifted $\frac{n}{2}$ times. We note that other cases when $\{v_1,v_2\}\subseteq S$, $v_1\in S$ and $v_2\in \overline{S}$ and $\{v_1,v_2\}\subseteq \overline{S}$ can be found in the shifted cases.

\medskip
\emph{Case~3.2. $n \equiv 6 \, (\modo \, 8)$.} In this case, $k$ is odd. We consider the case when $\{v_1,v_2\} \subseteq S$ as shown in Figure \ref{cycle6mod8}. By similar arguments as in Case~3.1, in general we have $2k+\frac{k-1}{2}=\frac{5k-1}{2}$ cases which can be shifted $n$ times and one case which can be shifted $\frac{n}{2}$ times.

Hence, in both cases we have $N_{\esd}(C_{4k+2})=\frac{5}{8}n(n-2)$.

\begin{figure}[!h]
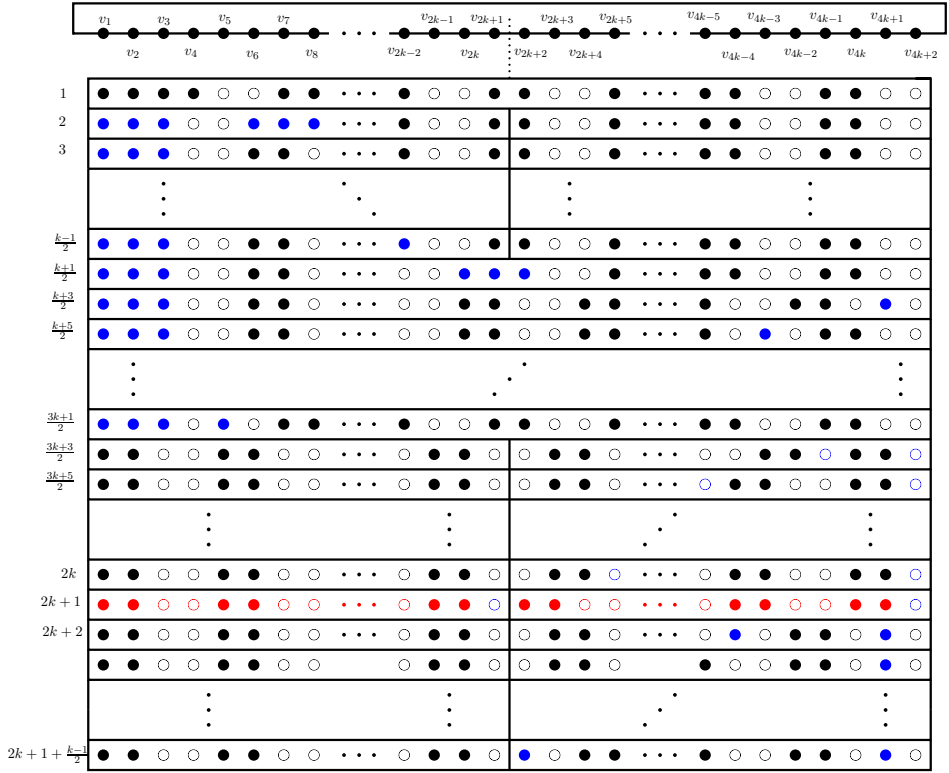

		\begin{center}
			\psscalebox{0.5 0.5}
{

}
		\end{center}
		\caption{Making the ESD-sets of cycle graph of order  $n \equiv 6 ~ (mod ~ 8)$ } \label{cycle6mod8}
	\end{figure}

\medskip
\emph{Case~4: $n=4k+3$, for some $k \ge 0$.} We observe that $N_{\esd}(C_{3}) = 3$, and so the desired result is immediate for $n=3$. Hence we may assume that $n \ge 7$. By Proposition~\ref{thm-special},  in this case we have  $|S|=2k+2$. So we need to choose $2k+2$ vertices in such a way as to yield a $\gamma_{\esd}$-set of $G$.

Suppose that $S$ contains three consecutive vertices on the cycle. Renaming vertices if necessary, we may assume that $\{v_1,v_2,v_3\} \subseteq S$. Letting $G' = G - v_2$ and $S' = S \setminus \{v_2\}$, we note that $G' = P_{4k+2}$ and $S'$ is an ESD-set of $G'$, implying by Theorem~\ref{thm-path-esp} that $2k + 2 = \gamma_{\esd}(G') \le |S'| = |S| - 1$, and so $|S| = 2k+3$, a contradiction. Hence, $S$ cannot contain three consecutive vertices on the cycle.

Suppose that $S$ contains a vertex with both its neighbors in $\overline{S}$. Renaming vertices if necessary, we may assume that $v_5 \in S$ and $\{v_4,v_6\} \subseteq \overline{S}$. In this case, $\{v_2,v_3,v_7,v_8\} \subseteq S$. Suppose that $v_9 \in S$. Letting $G' = G - \{v_3,v_4,\ldots,v_8\}$ and $S' = S \setminus \{v_3,v_5,v_7,v_8\}$, we note that $G' = P_{4k-3}$ and $S'$ is an ESD-set of $G'$, implying by Theorem~\ref{thm-path-esp} that $2k-1 = \gamma_{\esd}(G') \le |S'| = |S| - 4$, and so $|S| = 2k+3$, a contradiction. Hence, $v_9 \notin S$. By symmetry, $v_1 \notin S$. We now let $G'$ be obtained from $G - \{v_1,v_2,\ldots,v_9\}$ by adding the edge $v_{10}v_{4k+3}$. Further we let $S' = S \setminus \{v_2,v_3,v_5,v_7,v_8\}$. We note that $G' = C_{4k-6}$ and $S'$ is an ESD-set of $G'$, implying by Proposition \ref{thm-special} that $2k-2= \gamma_{\esd}(G') \le |S'| = |S| - 5$, and so $|S| = 2k+3$, a contradiction. Hence, $S$ does not contain a vertex with both its neighbors in $\overline{S}$, that is, $G[S]$ does not contain an isolated vertex.

By our earlier properties, we infer that every component of $G[S]$ is a $P_2$-component. Furthermore since $|S|=2k+2$, the subgraph $G[S]$ contains $k+1$ components. Therefore, the subgraph $G[\overline{S}]$ consists of $k+1$ components, where $k$ components are isomorphic to $P_2$ and one component isomorphic to $P_1$, that is, $G[\overline{S}] = P_1 \cup kP_2$. Since the isolated vertex in $G[\overline{S}]$ can be chosen in $n$ ways, this implies that $N_{\esd}(C_{4k+3})=n$.~\qed
\end{proof}

We close this section by counting the number, $N_{\esd}(P_n)$, of $\gamma_{\esd}$-sets in a path $P_n$ of order~$n$.

\begin{theorem}
\label{path-counting}
If $P_n$ is the path graph of order $n \ge 2$,  then
\begin{displaymath}
N_{\esd}(P_n)= \left\{ \begin{array}{ll}
1 & \textrm{if $n =4k$, } \1 \\
2k+1 & \textrm{if $n =4k+1$,} \1 \\
\frac{5k^2+5k+2}{2} & \textrm{if $n =4k+2$,} \1 \\
k+1 & \textrm{if $n =4k+3$.}
\end{array} \right.
\end{displaymath}
\end{theorem}	
\begin{proof}
For $n \ge 2$, let $G$ be a cycle $P_n$ given by $v_1v_2 \ldots v_n$, and let $S$ be a $\gamma_{\esd}$-set of $G$. As before, let $V = V(G)$. We consider four cases.

\medskip
\emph{Case~1: $n=4k$, for some $k \ge 1$.}  By Proposition~\ref{thm-special} and Theorem~\ref{thm-path-esp}, we have $\gamma_{\esd}(P_{4k}) = \gamma_{\esd}(C_{4k})$. In Case~1 of  Theorem~\ref{Cycle-counting}, we found all $\gamma_{\esd}$-set of $C_{4k}$. If we remove the edge $v_1v_{4k}$, then only $S=S_4$ is an ESD-set for $P_{4k}$ and we have $N_{\esd}(P_{4k})=1$.

\medskip
\emph{Case~2: $n=4k+1$, for some $k \ge 1$.} We proceed by induction on $k$. If $k=1$, then by Theorem~\ref{thm-path-esp}, we have $\gamma_{\esd}(P_{5})=3$. We have three $\gamma_{\esd}$-set for $P_5$, namely $S_1=\{v_1,v_2,v_5\}$, $S_2=\{v_1,v_3,v_5\}$ and $S_3=\{v_1,v_4,v_5\}$. Hence, $N_{\esd}(P_{5})=3$, as desired. This establishes the base case. Suppose that for $k=i$, we have $N_{\esd}(P_{4i+1})=2i+1$. We now count the number of $\gamma_{\esd}$-set for $P_{4(i+1)+1}$. Since $\gamma_{\esd}(P_{4i+5})=2i+3$, we need two more vertices in $S$ to produce a $\gamma_{\esd}$-set of $P_{4(i+1)+1}$. Clearly, in the last four vertices we should have more than one vertex in $S$, so these two vertices should be picked from $v_{4i+2}$, $v_{4i+3}$, $v_{4i+4}$ and $v_{4i+5}$. In all cases, $v_{4i+5} \in S$. Now, consider Figure~\ref{path4i+1-induction}, where the black vertices belong to $S$. In both rows (I) and (II) in Figure~\ref{path4i+1-induction}, we can only add $v_{4i+2}$ and $v_{4i+5}$ to $S$, but in row (III), we have three different choices for $S$, namely
\[
\begin{array}{lcl}
S_1 & = & S\cup\{v_{4i+2},v_{4i+5}\}, \1 \\
S_2 & = & S\cup\{v_{4i+3},v_{4i+5}\}, \1 \\
S_3 & = & S\cup\{v_{4i+4},v_{4i+5}\}.
\end{array}
\]

So, in total, we have $N_{\esd}(P_{4i+5})=N_{\esd}(P_{4i+1})+2$, and the induction is completed. One can easily check that there is no possibility that we have $v_{4i+1} \in \overline{S}$ and then find a $\gamma_{\esd}$-set for $P_{4i+5}$.

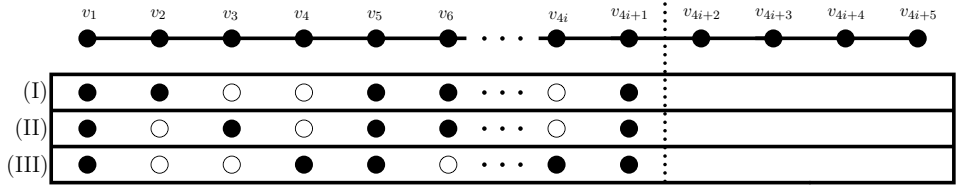
\begin{figure}
		\begin{center}
			\psscalebox{0.6 0.6}
{
\begin{pspicture}(0,-4.8)(21.21,-0.7)
\psdots[linecolor=black, dotsize=0.4](1.96,-1.55)
\psdots[linecolor=black, dotsize=0.4](3.56,-1.55)
\psdots[linecolor=black, dotsize=0.4](5.16,-1.55)
\psdots[linecolor=black, dotsize=0.4](6.76,-1.55)
\psdots[linecolor=black, dotsize=0.4](17.16,-1.55)
\psdots[linecolor=black, dotsize=0.4](18.76,-1.55)
\psdots[linecolor=black, dotsize=0.4](20.36,-1.55)
\rput[bl](1.84,-1.13){$v_1$}
\rput[bl](3.36,-1.13){$v_2$}
\rput[bl](4.96,-1.15){$v_3$}
\rput[bl](6.54,-1.15){$v_4$}
\psdots[linecolor=black, dotsize=0.1](10.76,-1.55)
\psdots[linecolor=black, dotsize=0.1](11.16,-1.55)
\psdots[linecolor=black, dotsize=0.1](11.56,-1.55)
\psline[linecolor=black, linewidth=0.08](20.36,-1.55)(16.76,-1.55)(16.76,-1.55)
\psline[linecolor=black, linewidth=0.08](5.16,-1.55)(1.96,-1.55)(1.96,-1.55)
\rput[bl](8.16,-1.15){$v_5$}
\rput[bl](9.74,-1.15){$v_6$}
\rput[bl](0.36,-3.01){\begin{Large}
(I)
\end{Large}}
\rput[bl](0.18,-3.83){\begin{Large}
(II)
\end{Large}}
\rput[bl](0.0,-4.65){\begin{Large}
(III)
\end{Large}}
\psline[linecolor=black, linewidth=0.08](17.96,-4.75)(17.96,-4.75)
\rput[bl](19.96,-1.15){$v_{4i+5}$}
\rput[bl](18.36,-1.15){$v_{4i+4}$}
\rput[bl](16.76,-1.15){$v_{4i+3}$}
\psline[linecolor=black, linewidth=0.08](10.36,-1.55)(5.16,-1.55)(5.16,-1.55)
\psdots[linecolor=black, dotsize=0.4](8.36,-1.55)
\psdots[linecolor=black, dotsize=0.4](9.96,-1.55)
\psdots[linecolor=black, dotsize=0.4](15.56,-1.55)
\psdots[linecolor=black, dotsize=0.4](13.96,-1.55)
\psline[linecolor=black, linewidth=0.08](13.56,-1.55)(17.16,-1.55)(17.16,-1.55)
\rput[bl](15.16,-1.15){$v_{4i+2}$}
\rput[bl](13.56,-1.15){$v_{4i+1}$}
\psdots[linecolor=black, dotsize=0.4](12.36,-1.55)
\psline[linecolor=black, linewidth=0.08](11.96,-1.55)(13.96,-1.55)
\psline[linecolor=black, linewidth=0.08](13.96,-1.55)(13.96,-1.55)
\rput[bl](12.18,-1.19){$v_{4i}$}
\psline[linecolor=black, linewidth=0.08](1.56,-2.35)(21.16,-2.35)(21.16,-4.75)(1.16,-4.75)(1.16,-2.35)(1.56,-2.35)
\psline[linecolor=black, linewidth=0.08](1.16,-3.15)(21.16,-3.15)(21.16,-3.15)
\psline[linecolor=black, linewidth=0.08](1.16,-3.95)(21.16,-3.95)(21.16,-3.95)
\psline[linecolor=black, linewidth=0.08, linestyle=dotted, dotsep=0.10583334cm](14.76,-0.75)(14.76,-4.75)
\psdots[linecolor=black, dotsize=0.4](1.96,-2.75)
\psdots[linecolor=black, dotsize=0.4](1.96,-3.55)
\psdots[linecolor=black, dotsize=0.4](1.96,-4.35)
\psdots[linecolor=black, dotsize=0.4](13.96,-2.75)
\psdots[linecolor=black, dotsize=0.4](13.96,-3.55)
\psdots[linecolor=black, dotsize=0.4](13.96,-4.35)
\psdots[linecolor=black, dotsize=0.4](3.56,-2.75)
\psdots[linecolor=black, dotsize=0.4](5.16,-3.55)
\psdots[linecolor=black, dotsize=0.4](6.76,-4.35)
\psdots[linecolor=black, dotsize=0.4](8.36,-4.35)
\psdots[linecolor=black, dotsize=0.4](8.36,-2.75)
\psdots[linecolor=black, dotsize=0.4](9.96,-2.75)
\psdots[linecolor=black, dotsize=0.4](8.36,-3.55)
\psdots[linecolor=black, dotsize=0.4](9.96,-3.55)
\psdots[linecolor=black, dotstyle=o, dotsize=0.4, fillcolor=white](5.16,-2.75)
\psdots[linecolor=black, dotstyle=o, dotsize=0.4, fillcolor=white](6.76,-2.75)
\psdots[linecolor=black, dotstyle=o, dotsize=0.4, fillcolor=white](3.56,-3.55)
\psdots[linecolor=black, dotstyle=o, dotsize=0.4, fillcolor=white](6.76,-3.55)
\psdots[linecolor=black, dotstyle=o, dotsize=0.4, fillcolor=white](3.56,-4.35)
\psdots[linecolor=black, dotstyle=o, dotsize=0.4, fillcolor=white](5.16,-4.35)
\psdots[linecolor=black, dotstyle=o, dotsize=0.4, fillcolor=white](9.96,-4.35)
\psdots[linecolor=black, dotstyle=o, dotsize=0.4, fillcolor=white](12.36,-2.75)
\psdots[linecolor=black, dotstyle=o, dotsize=0.4, fillcolor=white](12.36,-3.55)
\psdots[linecolor=black, dotsize=0.4](12.36,-4.35)
\psdots[linecolor=black, dotsize=0.1](10.76,-2.75)
\psdots[linecolor=black, dotsize=0.1](11.16,-2.75)
\psdots[linecolor=black, dotsize=0.1](11.56,-2.75)
\psdots[linecolor=black, dotsize=0.1](10.76,-3.55)
\psdots[linecolor=black, dotsize=0.1](11.16,-3.55)
\psdots[linecolor=black, dotsize=0.1](11.56,-3.55)
\psdots[linecolor=black, dotsize=0.1](10.76,-4.35)
\psdots[linecolor=black, dotsize=0.1](11.16,-4.35)
\psdots[linecolor=black, dotsize=0.1](11.56,-4.35)
\end{pspicture}
}
		\end{center}
		\caption{Making the ESD-sets of path graph of order $4k+1$} \label{path4i+1-induction}
	\end{figure}

\medskip
\emph{Case~3: $n=4k+2$, for some $k \ge 0$.} By Proposition \ref{thm-special} and Theorem \ref{thm-path-esp}, we have $\gamma_{\esd}(P_{4k+2})=\gamma_{\esd}(C_{4k+2})$. We now proceed in a similar way as in Case~1. We have two subcases:

\medskip
\emph{Case~3.1: $k$ is odd.} Here we consider Figure~\ref{cycle6mod8} and find all $\gamma_{\esd}$-sets of this cycle. Removing the edge $v_1v_{4k+2}$ yields a path. First, we consider row 1. In this row, we have a set of four consecutive vertices in $S$, and $k-1$ sets of two consecutive vertices in $S$, which in total gives us $4k+2$ sets after shifting vertices, as we did in proof of Theorem~\ref{Cycle-counting}. Among these sets, we have three sets with four consecutive vertices in $S$, and these sets contain both $v_1$ and $v_{4k+2}$, and $k-1$ sets of two consecutive vertices. So, in total for row 1, we have $k+2$ $\gamma_{\esd}$-set for $P_{4k+2}$. In the next step, we do the same for all rows in Figure \ref{cycle6mod8}. So, in general, we have:
\begin{itemize}
\item[(a)]
One row (row 1) with a set of four consecutive vertices and $k-1$ sets of two consecutive vertices in $S$, which yields $k+2$ $\gamma_{\esd}$-sets.
\item[(b)]
$\frac{k-1}{2}$ rows (rows 2 to $\frac{k+1}{2}$) with two sets of three consecutive vertices and $k-2$ sets of two consecutive vertices in $S$, which in each row, yields $k+2$ $\gamma_{\esd}$-sets.
\item[(c)]
$k$ rows (rows $\frac{k+3}{2}$ to $\frac{3k+1}{2}$) with a set of three consecutive vertices, a set with single vertex and $k-1$ sets of two consecutive vertices in $S$, which in each row, yields $k+1$ $\gamma_{\esd}$-sets.
\item[(d)]
$\frac{k-1}{2}$ rows (rows $\frac{3k+3}{2}$ to $2k$) with  $k+1$ sets of two consecutive vertices in $S$, which  in each row, yields $k+1$ $\gamma_{\esd}$-sets.
\item[(e)]
One row (row $2k+1$) with $k+1$ sets of two consecutive vertices in $S$, which yields $\frac{k+1}{2}$ $\gamma_{\esd}$-sets. (Note that in this row which is separated by red color, half of the cases occur twice.)
\item[(f)]
$\frac{k-1}{2}$ rows (rows $2k+2$ to $2k+1+\frac{k-1}{2}$) with two sets of single vertex and $k$ sets of two consecutive vertices in $S$, which  in each row, yields $k$ $\gamma_{\esd}$-sets.
\end{itemize}
By summing all these cases, we have $\frac{5k^2+5k+2}{2}$ $\gamma_{\esd}$-sets for $P_{4k+2}$.

\medskip
\emph{Case~3.2: $k$ is even.}  By considering Figure \ref{cycle2mod8}, analogous argument as in Case~3.1 yield the desired result.

\medskip
\emph{Case~3: $n=4k+32$, for some $k \ge 0$.} using analogous arguments as in Case~2 and using Figure \ref{path4i+3-induction}, we infer that there is only one more set which is added to row (I). This completes the proof.~\qed
\end{proof}

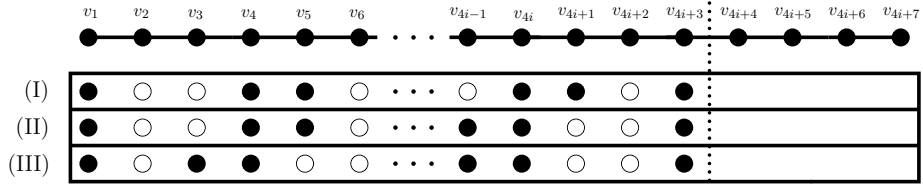
\begin{figure}
		\begin{center}
			\psscalebox{0.6 0.6}
{
\begin{pspicture}(0,-4.8)(20.41,-0.7)
\psdots[linecolor=black, dotsize=0.4](1.96,-1.55)
\psdots[linecolor=black, dotsize=0.4](3.16,-1.55)
\psdots[linecolor=black, dotsize=0.4](4.36,-1.55)
\psdots[linecolor=black, dotsize=0.4](5.56,-1.55)
\psdots[linecolor=black, dotsize=0.4](13.96,-1.55)
\psdots[linecolor=black, dotsize=0.4](15.16,-1.55)
\psdots[linecolor=black, dotsize=0.4](16.36,-1.55)
\rput[bl](1.84,-1.13){$v_1$}
\rput[bl](2.96,-1.13){$v_2$}
\rput[bl](4.16,-1.15){$v_3$}
\rput[bl](5.34,-1.15){$v_4$}
\psdots[linecolor=black, dotsize=0.1](8.76,-1.55)
\psdots[linecolor=black, dotsize=0.1](9.16,-1.55)
\psdots[linecolor=black, dotsize=0.1](9.56,-1.55)
\psline[linecolor=black, linewidth=0.08](18.36,-1.55)(14.76,-1.55)(14.76,-1.55)
\psline[linecolor=black, linewidth=0.08](5.16,-1.55)(1.96,-1.55)(1.96,-1.55)
\rput[bl](6.56,-1.15){$v_5$}
\rput[bl](7.74,-1.15){$v_6$}
\rput[bl](0.36,-3.01){\begin{Large}
(I)
\end{Large}}
\rput[bl](0.18,-3.83){\begin{Large}
(II)
\end{Large}}
\rput[bl](0.0,-4.65){\begin{Large}
(III)
\end{Large}}
\psline[linecolor=black, linewidth=0.08](17.96,-4.75)(17.96,-4.75)
\rput[bl](17.16,-1.15){$v_{4i+5}$}
\rput[bl](15.96,-1.15){$v_{4i+4}$}
\rput[bl](14.76,-1.15){$v_{4i+3}$}
\psdots[linecolor=black, dotsize=0.4](6.76,-1.55)
\psdots[linecolor=black, dotsize=0.4](7.96,-1.55)
\psdots[linecolor=black, dotsize=0.4](12.76,-1.55)
\psdots[linecolor=black, dotsize=0.4](11.56,-1.55)
\psline[linecolor=black, linewidth=0.08](11.56,-1.55)(15.16,-1.55)(15.16,-1.55)
\rput[bl](13.56,-1.15){$v_{4i+2}$}
\rput[bl](12.36,-1.15){$v_{4i+1}$}
\psdots[linecolor=black, dotsize=0.4](10.36,-1.55)
\psline[linecolor=black, linewidth=0.08](9.96,-1.55)(11.96,-1.55)
\psline[linecolor=black, linewidth=0.08](13.96,-1.55)(13.96,-1.55)
\rput[bl](11.38,-1.19){$v_{4i}$}
\psline[linecolor=black, linewidth=0.08, linestyle=dotted, dotsep=0.10583334cm](15.72,-0.75)(15.72,-4.75)
\psdots[linecolor=black, dotsize=0.4](1.96,-2.75)
\psdots[linecolor=black, dotsize=0.4](1.96,-3.55)
\psdots[linecolor=black, dotsize=0.4](1.96,-4.35)
\psdots[linecolor=black, dotsize=0.4](15.16,-2.75)
\psdots[linecolor=black, dotsize=0.4](15.16,-3.55)
\psdots[linecolor=black, dotsize=0.4](15.16,-4.35)
\psdots[linecolor=black, dotstyle=o, dotsize=0.4, fillcolor=white](3.16,-2.75)
\psdots[linecolor=black, dotsize=0.4](5.56,-3.55)
\psdots[linecolor=black, dotsize=0.4](4.36,-4.35)
\psdots[linecolor=black, dotsize=0.4](5.56,-4.35)
\psdots[linecolor=black, dotsize=0.4](6.76,-2.75)
\psdots[linecolor=black, dotsize=0.4](6.76,-3.55)
\psdots[linecolor=black, dotsize=0.4](11.56,-2.75)
\psdots[linecolor=black, dotstyle=o, dotsize=0.4, fillcolor=white](4.36,-2.75)
\psdots[linecolor=black, dotsize=0.4](5.56,-2.75)
\psdots[linecolor=black, dotstyle=o, dotsize=0.4, fillcolor=white](3.16,-3.55)
\psdots[linecolor=black, dotstyle=o, dotsize=0.4, fillcolor=white](4.36,-3.55)
\psdots[linecolor=black, dotstyle=o, dotsize=0.4, fillcolor=white](3.16,-4.35)
\psdots[linecolor=black, dotstyle=o, dotsize=0.4, fillcolor=white](6.76,-4.35)
\psdots[linecolor=black, dotstyle=o, dotsize=0.4, fillcolor=white](7.96,-4.35)
\psdots[linecolor=black, dotstyle=o, dotsize=0.4, fillcolor=white](7.96,-2.75)
\psdots[linecolor=black, dotstyle=o, dotsize=0.4, fillcolor=white](7.96,-3.55)
\psdots[linecolor=black, dotsize=0.4](10.36,-3.55)
\psdots[linecolor=black, dotsize=0.1](8.76,-2.75)
\psdots[linecolor=black, dotsize=0.1](9.16,-2.75)
\psdots[linecolor=black, dotsize=0.1](9.56,-2.75)
\psdots[linecolor=black, dotsize=0.1](8.76,-3.55)
\psdots[linecolor=black, dotsize=0.1](9.16,-3.55)
\psdots[linecolor=black, dotsize=0.1](9.56,-3.55)
\psdots[linecolor=black, dotsize=0.1](8.76,-4.35)
\psdots[linecolor=black, dotsize=0.1](9.16,-4.35)
\psdots[linecolor=black, dotsize=0.1](9.56,-4.35)
\psline[linecolor=black, linewidth=0.08](8.36,-1.55)(5.16,-1.55)(5.16,-1.55)
\psdots[linecolor=black, dotsize=0.4](12.76,-2.75)
\psdots[linecolor=black, dotstyle=o, dotsize=0.4, fillcolor=white](13.96,-2.75)
\psdots[linecolor=black, dotstyle=o, dotsize=0.4, fillcolor=white](10.36,-2.75)
\psdots[linecolor=black, dotstyle=o, dotsize=0.4, fillcolor=white](13.96,-3.55)
\psdots[linecolor=black, dotstyle=o, dotsize=0.4, fillcolor=white](12.76,-3.55)
\psdots[linecolor=black, dotstyle=o, dotsize=0.4, fillcolor=white](13.96,-4.35)
\psdots[linecolor=black, dotstyle=o, dotsize=0.4, fillcolor=white](12.76,-4.35)
\psdots[linecolor=black, dotsize=0.4](11.56,-3.55)
\psdots[linecolor=black, dotsize=0.4](11.56,-4.35)
\psdots[linecolor=black, dotsize=0.4](10.36,-4.35)
\psdots[linecolor=black, dotsize=0.4](17.56,-1.55)
\psdots[linecolor=black, dotsize=0.4](18.76,-1.55)
\psdots[linecolor=black, dotsize=0.4](19.96,-1.55)
\psline[linecolor=black, linewidth=0.08](1.96,-2.35)(20.36,-2.35)(20.36,-4.75)(1.56,-4.75)(1.56,-2.35)(1.96,-2.35)
\psline[linecolor=black, linewidth=0.08](1.56,-3.15)(20.36,-3.15)(20.36,-3.15)
\psline[linecolor=black, linewidth=0.08](1.56,-3.95)(20.36,-3.95)(20.36,-3.95)
\rput[bl](9.96,-1.15){$v_{4i-1}$}
\psline[linecolor=black, linewidth=0.08](18.36,-1.55)(19.96,-1.55)(19.96,-1.55)
\rput[bl](18.36,-1.15){$v_{4i+6}$}
\rput[bl](19.56,-1.15){$v_{4i+7}$}
\end{pspicture}
}
		\end{center}
		\caption{Making the ESD-sets of path graph of order $4k+3$} \label{path4i+3-induction}
	\end{figure}

\section{Acknowledgements}

The second author would like to thank the Research Council of Norway and Department of Informatics, University of Bergen for their support.

\end{document}